\documentclass[11 pt,fullpage]{article}
\usepackage{amsfonts,amssymb}
\usepackage[hidelinks]{hyperref}
\usepackage{lipsum,amsmath,amsthm}
\usepackage[numbers,sort&compress]{natbib}
\usepackage{ dsfont }
\usepackage{appendix}
\usepackage{color}
\usepackage{esint}
\usepackage{hyperref}
\usepackage{mathtools}
\usepackage{scalerel}
\usepackage{comment}
\mathtoolsset{showonlyrefs}
\usepackage[right = 2.5cm, left=2.5cm, top = 2.5cm, bottom =2.5cm]{geometry}
\usepackage[utf8]{inputenc}
\usepackage[english]{babel}

\theoremstyle{plain}
\newtheorem{theorem}{Theorem}

\newtheorem{proposition}{Proposition}[section]
\newtheorem{lemma}[proposition]{Lemma}

\newtheorem{definition}{Definition}[section]

\theoremstyle{definition}
\newtheorem{remark}{Remark}[section]

\usepackage{todonotes}
\newcommand{\one}{\mathbf{1}}

\allowdisplaybreaks

\numberwithin{equation}{section}

\newcommand{\grad}{\nabla}

\newcommand{\R}{\mathbb{R}}
\newcommand{\N}{\mathbb{N}}

\newcommand{\Z}{\mathbb{Z}}
\newcommand{\T}{\mathbb{T}}

\newcommand{\dee}{\mathrm{d}}
\newcommand{\ds}{\dee s}
\newcommand{\dt}{\dee t}

\newcommand{\dx}{\mathrm{d} x}

\newcommand{\dz}{\dee z}

\usepackage{bbm}

\renewcommand{\P}{\mathbf{P}}

\newcommand{\EE}{\mathbf E}

\newtheorem*{lemma*}{Lemma}

\theoremstyle{definition}

 \def\Xint#1{\mathchoice
       {\XXint\displaystyle\textstyle{#1}}%
       {\XXint\textstyle\scriptstyle{#1}}%
       {\XXint\scriptstyle\scriptscriptstyle{#1}}%
       {\XXint\scriptscriptstyle\scriptscriptstyle{#1}}%
       \!\int}
    \def\XXint#1#2#3{{\setbox0=\hbox{$#1{#2#3}{\int}$}
         \vcenter{\hbox{$#2#3$}}\kern-.5\wd0}}
    
    \def\dashint{\Xint-}

\DeclareUnicodeCharacter{2BC}{'}

\usepackage{chngcntr}
\counterwithout{equation}{subsection}
\counterwithout{equation}{section} 

\begin{document}
	
	\title{Optimal enhanced dissipation and mixing for a time-periodic, Lipschitz velocity field on $\mathbb{T}^2$} 
	\author{Tarek M. Elgindi, Kyle Liss, and Jonathan C. Mattingly}
	
	\maketitle
	
	\begin{abstract}
 We consider the advection-diffusion equation on $\T^2$ with a Lipschitz and time-periodic velocity field that alternates between two piecewise linear shear flows. We prove enhanced dissipation on the timescale $|\log \nu|$, where $\nu$ is the diffusivity parameter. This is the optimal decay rate as $\nu \to 0$ for uniformly-in-time Lipschitz velocity fields. We also establish exponential mixing for the $\nu = 0$ problem.  
	\end{abstract}
	
	\setcounter{tocdepth}{1}
	{\small\tableofcontents}
	
	\section{Introduction} \label{sec:intro}

	Understanding the dynamics of a diffusive scalar  advected by an incompressible flow is a fundamental problem in fluid mechanics with relevance to both engineering and the natural sciences. It is often possible in applications to treat the scalar as a \textit{passive tracer}; i.e., to assume it has no feedback on the fluid flow. The relevant model is then the advection-diffusion equation
\begin{equation} 
\label{eq:advectiondiffusion}
\begin{cases}
\partial_t f + u \cdot \grad f = \nu \Delta f, \\ 
f|_{t=0} = f_0
\end{cases}
\end{equation} 
posed on a suitable spatial domain, which we will always take to be the two-dimensional torus $\T^2 = \R^2 \setminus \Z^2$. Here, $0 \le \nu \ll 1$ is a fixed diffusivity constant, $u :[0,\infty) \times \T^2 \to \R^2$ is a
given divergence-free velocity
field, and $f :[0,\infty) \times \T^2 \to \R$ is the scalar
unknown. In applications, $f$ may represent, for example,
temperature or the concentration of a chemical in water. 
The primary purpose of this paper is to study the quantitative decay rate of solutions with $\nu > 0$ for a particular velocity field $u$, which we also show defines an optimally mixing dynamical system on $\T^2$. 

We first recall that since $u$ is divergence free, the spatial average
of any solution to \eqref{eq:advectiondiffusion} is conserved and in
fact all $L^2$ solutions with $\nu > 0$ converge strongly to their
spatial average as $t \to \infty$. Indeed, a simple energy estimate
using just the incompressibility of $u$ gives 
\begin{equation} \label{eq:heatwgrad}
\frac{1}{2}\frac{d}{dt}\left\|f(t)-\dashint f_0 \right\|_{L^2}^2 = -\nu\|\grad f(t)\|_{L^2}^2,
\end{equation}
where $\dashint f_0=\frac{1}{\text{Vol}(\T^2)} \int_{\T^2} f_0(z) \dz$ denotes the spatial average. Hence, the Poincar\'{e} inequality implies that
\begin{equation}\label{eq:heat}
\left\|f(t)-\dashint f_0 \right\|_{L^2} \le e^{-\nu t}\left\|f_0 - \dashint f_0\right\|_{L^2}.
\end{equation}
When $u \equiv 0$, reducing \eqref{eq:advectiondiffusion} to the diffusion/heat equation, each nonzero Fourier mode of the solution decays independently at a frequency-dependent rate and the $\nu^{-1}$ decay timescale of \eqref{eq:heat} is in general sharp. An intriguing feature of \eqref{eq:advectiondiffusion} is that when diffusion and advection are both present, the formation of small scales induced by stirring can cause solutions to dissipate energy much faster than the heat equation as $\nu \to 0$, greatly accelerating the convergence in \eqref{eq:heat}. This phenomenon is often referred to as \textit{enhanced dissipation}. Intuitively, advection results in a conservative transfer of energy to high frequencies, which the diffusion damps much more effectively. Enhanced dissipation can be defined precisely as solutions to \eqref{eq:advectiondiffusion} exhibiting a decay timescale that is $o(\nu^{-1})$ as $\nu \to 0$.

\begin{definition}\label{EnhancedDissipation} 
A divergence-free and time-dependent velocity field $u:[0,\infty)\times \mathbb{T}^2\to\mathbb{R}^2$ is said to be {\bf dissipation enhancing} if there exists a constant $C\geq 1$ and a rate function $\delta(\nu):(0,\infty)\rightarrow (0,\infty)$ with \[\lim_{\nu\rightarrow 0}\frac{\nu}{\delta(\nu)}=0\] so that for all $f_0 \in L^2$ the solution $f$ of \eqref{eq:advectiondiffusion} satisfies 
  \begin{equation}
  \left\|f(t)-\dashint f_0\right\|_{L^2}\leq C e^{-\delta(\nu) t} \left\|f_0-\dashint f_0\right\|_{L^2}
  \end{equation} 
  for all $t\geq 0$.
	\end{definition}

\noindent While enhanced dissipation is well understood in certain cases, e.g. shear flows \cite{BCZ17, ANB22, Wei21, He22, CZD20, CZDr19, GCZ21}, the effects of diffusion and advection naturally do not commute and understanding the interplay between them in full generality is extremely challenging. In fact, even though it is straightforward to prove that the optimal rate function for a uniformly-in-time Lipschitz velocity field is $\delta(\nu) \approx |\log \nu|$, the only flows we are aware of that achieve this rate are those constructed in \cite{BBPSdiss} from generic solutions to various stochastically forced fluid models.

Our main goal in this paper is to introduce an explicit, time-periodic
velocity field that is uniformly-in-time Lipschitz and whose dynamics we can understand deeply enough to establish \textit{sharp} dissipation enhancement estimates. We consider in particular alternating piecewise linear shear flows. For $\alpha > 0$, let $H_\alpha:\T^2 \to \R^2$ and $V_\alpha:\T^2 \to \R^2$ (with $\T \cong [0,1])$ denote the shear flows
	\begin{equation}
	H_\alpha(x,y) = \begin{pmatrix}
	-2\alpha |y-1/2| \\ 0 
	\end{pmatrix} \quad \text{and} \quad 
	V_\alpha(x,y) = \begin{pmatrix}
	0 \\ -2\alpha |x-1/2|
	\end{pmatrix}.
	\end{equation}
	We then define the Lipschitz continuous (in space), divergence-free velocity field $u_\alpha: [0,\infty) \times \T^2 \to \R^2$ by
	\begin{equation} \label{eq:ualpha}
	u_\alpha(t,x,y) = \begin{cases} 
	V_\alpha(x,y)  & t\in[0,1/2) \\
	H_\alpha(x,y) & t\in [1/2,1)
	\end{cases}
	\end{equation}
	for $t \in [0,1)$ and extended for $t \ge 1$ to be time periodic with a period of one. An important feature of $u_\alpha$ is that it generates a \textit{uniformly hyperbolic} map on $\T^2$ with singularities for $\alpha$ large. We mention that variations on the velocity field $u_\alpha$ have been considered in the recent works \cite{ZlatosNumerical, Hill22, Hill23}, which will be discussed more in Section~\ref{sec:literature} below.
 
 Our main result is stated as follows.
	
\begin{theorem} \label{thrm:enhanced} 
If $\alpha$ is a sufficiently large even integer, then the velocity
field $u_\alpha$ is dissipation enhancing with the optimal rate
function $$\delta(\nu) = \frac{c}{|\log \nu|},$$
where $c > 0$ is a constant depending on $\alpha$ but not $\nu$. 
\end{theorem}

\noindent The fact that the decay timescale in Theorem~\ref{thrm:enhanced} is optimal with respect to scaling in $\nu$ for Lipschitz $u$ can be proven easily with \eqref{eq:heatwgrad} and the fact that $\|\grad f(t)\|_{L^2}$ can grow at most exponentially fast.

The estimate in \eqref{eq:heat} relies only on the dissipation produced by the Laplacian and ignores the effects of the transport term. One must leverage the mixing produced from the advection by $u$ to obtain the optimal enhanced dissipation in Theorem~\ref{thrm:enhanced}. To this end, we study the large scale mixing properties of
dynamics which hold for $\nu=0$. 
Mixing refers, roughly speaking, to chaotic stretching and folding. An initially localized concentration of scalar tends to filament and spread throughout the entire domain, resulting in the conservative transfer of energy to higher Fourier modes. Precisely, a divergence-free and time-dependent velocity field $u:[0,\infty)\times \mathbb{T}^2\to\mathbb{R}^2$ is said to be \textit{mixing} if for every $f_0 \in L^2$ the solution to the advection equation 
\begin{equation}
	\label{eq:advection}
	\begin{cases} 
	 \partial_t f+u\cdot\nabla f=0, \\ 
	 f|_{t=0} = f_0
	 \end{cases}
	\end{equation}
 converges weakly in $L^2$ as $t \to \infty$ to its spatial average. 

As a corollary to the proof of Theorem~\ref{thrm:enhanced} and an existing result in hyperbolic dynamics \cite{DemersLiverani08}, we are able to obtain optimal mixing estimates for the velocity field $u_\alpha$. To state this result we begin by recalling a quantitative notion of mixing rate. For a given divergence-free velocity field $u$, let $\phi_t:\mathbb{T}^2\rightarrow\mathbb{T}^2$ denote the associated flow map, which solves
 \begin{equation}
   \label{eq:flowmap}
\left\{\begin{aligned}
 &\tfrac{d}{dt} \phi_t(z) =u(t,\phi_t(z)), \\
&\phi_0 (z)=z \in \mathbb{T}^2.
 \end{aligned}\right.
 \end{equation}
 For any Lipschitz $u$, $\phi_t$ is well defined and a (Lebesgue)
 measure-preserving homeomorphism for each $t \ge 0$. The solution to
 the pure advection equation \eqref{eq:advection} is then given by \begin{equation} \label{eq:transportsol}
 f(t) = f_0 \circ \phi_t^{-1}.
 \end{equation}
From \eqref{eq:transportsol}, one can naturally quantify mixing in terms of correlation decay of sufficiently regular observables for the dynamical system on $\T^2$ generated by $\phi_t$. We are interested here in \textit{exponential mixing}, as this is easily seen to be the optimal rate for uniformly-in-time Lipschitz velocity fields.

\begin{definition}\label{ExponentialMixing}
		A Lipschitz and divergence-free velocity field $u$ is said to be an {\bf exponential mixer} if there exist $C, c>0$ so that
		$$ \left|\;\dashint (f \circ \phi_t^{-1})\, g - \dashint f \;\dashint g\;\right| \le C\exp(-c t)\|f\|_{C^1}\|g\|_{C^1} $$
   for any $f,g\in C^1(\mathbb{T}^2)$. This rate is optimal for uniformly-in-time Lipschitz $u$.
	\end{definition}

 \begin{remark}
     We mention that while there are many simple examples of exponentially mixing \textit{maps}, for example Arnold's cat map and the Baker's map, exponentially mixing \textit{flows} as we consider here are harder to construct.
 \end{remark}

\begin{remark} \label{rem:rectanglemix}
Exponential mixing in the sense of Definition~\ref{ExponentialMixing} is equivalent to the existence of constants $C, c > 0$ such that for any $N \in \N$ and squares $R,Q \subseteq \T^2$ of side length $2^{-N}$ there holds
\begin{equation} \label{eq:rectanglemix}
|\,m(\phi_n^{-1}(R) \cap Q) - m(R)\,m(Q)\,| \le Ce^{-cn}
\end{equation}
for all $n \ge CN$, where $m$ denotes Lebesgue measure on $\T^2$. In other words, exponential mixing means that any square at a given length scale $2^{-N}$ that is pushed through the dynamics becomes roughly evenly distributed over $\T^2$ at the spatial resolution $2^{-N}$ after time approximately $N$. This is the main perspective on mixing and its relationship to dissipation enhancement that we will utilize in this paper.
\end{remark}

We can now state our mixing result for $u_\alpha$.

\begin{theorem} \label{thrm:mix}
    If $\alpha$ is a sufficiently large even integer, then $u_\alpha$ is exponentially mixing in the sense of Definition~\ref{ExponentialMixing}.
    \end{theorem}

\begin{remark}
The restriction that $\alpha$ be an even integer is not important for the proof of either Theorem~\ref{thrm:enhanced} or Theorem~\ref{thrm:mix}. It is merely done for convenience so that the time-1 flow map of $u_\alpha$ can be expressed in a simple way as a piecewise toral automorphism (see the beginning of Section~\ref{sec:hyperbolic}). On the other hand, the assumption that $\alpha$ be sufficiently large is necessary, at least for Theorem~\ref{thrm:mix}. Indeed, the recent work \cite{Hill23} considers the discrete time map defined by $u_\alpha$ for $\alpha$ small and shows that in this case correlations decay only at a polynomial rate. This also suggests that $u_\alpha$ can be dissipation enhancing with a rate function that is at best polynomial in $\nu$ for small $\alpha$, but such an implication is not strictly true in general.
\end{remark}

\begin{remark}
For both of our results, the velocity field can be made smooth in time with the proof unchanged by replacing $H_\alpha$ and $V_\alpha$ in the definition of $u_\alpha$ with $\varphi(t)H_1$ and $\varphi(1/2-t)V_1$ for any smooth function $\varphi:[0,1/2] \to [0,\infty)$ with compact support in $(0,1/2)$ satisfying $\int_0^{1/2}\varphi(t) \dt = \alpha$.
\end{remark}

We emphasize here that while the properties of the $\nu = 0$ dynamics are fundamental to obtaining enhanced dissipation on the $|\log \nu|$ timescale, exponential mixing is not known to imply sharp dissipation enhancement for the $\nu > 0$ problem in general. There is a general result
\cite{CZDr19} that one can take $\delta(\nu)= C |\log\nu|^{-2}$ for \emph{any} Lipschitz
continuous exponentially mixing flow. This is not quite optimal because of the 2 in the exponent of the logarithm, which seems very hard to remove in general because of difficulties with approximating the $\nu > 0$ dynamics by the $\nu = 0$ problem. In our proof of Theorem~\ref{thrm:enhanced} we consider the $\nu > 0$ dynamics directly from the viewpoint of Markov processes and the stochastic representation of  \eqref{eq:advectiondiffusion} and do not explicitly use exponential mixing. We rely only on proving that a fixed fraction of mass of any initial scalar is mixed to the diffusive $\sqrt{\nu}$ scale after time $t \approx |\log \nu|$ and moreover that this mixing crucially persists in some sense for the $\nu > 0$ dynamics. This is described by Lemma~\ref{lem:weakmix}, which we are able to prove in a direct and self-contained way. 

\subsection{Previous results} \label{sec:literature}
The study of mixing and dissipation enhancing flows has seen an explosion of work in the last few decades. In this section, we will outline the existing works and then discuss how the present contribution fits into the picture.

\subsubsection{PDE Literature}

Early work on enhanced dissipation was done by Kelvin \cite{Kelvin} and Kolmogorov \cite{Kolmogorov34}. An important foundational work is the paper \cite{CKRZ08} in which the authors establish, as a corollary of a more general result, necessary and sufficient conditions for \emph{autonomous} flows to be dissipation enhancing. Extensions to the case of time-periodic flows were established in \cite{KZS08,Zlatos10}. Building off of \cite{CKRZ08}, the authors of \cite{CZED20,FengIyer19} established quantitative enhanced dissipation rates based on quantitative mixing rates. A restricted form of enhanced dissipation has also been established in a number of specific settings like for shear flows \cite{BCZ17, Wei21,ANB22, He22, WCZC20, GCZ21} in various domains and using a variety of proofs as well as circular flows \cite{FengMazz22,CZD20}, cellular flows \cite{IyerZhou22, IyerXuZlatos21}, and general flows with autonomous Hamiltonian \cite{V21,BCZM22}. Enhanced dissipation for certain shear flows has also been established in various nonlinear settings \cite{CZEW20,WeiZhang19,MasmoudiZhao19, BGM17, WeiZhangZhao20, BGM20}. 

A number of works on mixing in the PDE literature have also appeared recently in relation to Bressan's mixing conjecture \cite{Bressan03}. There, a central question is whether \emph{low regularity} velocity fields can mix faster than exponentially. This is still open in general (in particular in the setting of Bressan's conjecture), though significant progress has been made both on the positive side \cite{CripDeLellis08, IyKis14, Cooperman22} and the negative side \cite{ElgindiZlatos19, YaoZlatos17,ACM19} . For additional works focused on optimal mixing subject to certain physical constraints on the flow, see e.g. \cite{Lin11,Lunasin12}.

\subsubsection{Chaotic Dynamics Literature}

The invariant manifold structure and statistical properties of hyperbolic maps have been studied extensively in the dynamics literature \cite{Liverani95, Chernov92, ChernovZhang05, Chernov99, BuniChernovSinai91, Young98, KatokStrelcyn,DemersLiverani08, BaladiDemersLiverani18}. Exponential decay of correlations is known in various situations, including certain examples of billiards \cite{Young98, Chernov99, BaladiDemersLiverani18} and some piecewise linear maps \cite{Chernov92, Hill22}. Mixing has also seen new developments in the context of \textit{random} dynamical systems. The first examples of spatially smooth exponentially mixing flows on the torus were constructed in \cite{BBPSchaos,BBPSmix} from solutions to various stochastic fluid models with sufficiently non-degenerate noise. These flows are also known to give enhanced dissipation on the $|\log \nu|$ timescale \cite{BBPSdiss}. Note that the velocity fields studied in \cite{BBPSchaos, BBPSdiss, BBPSmix} are smooth but are not uniformly bounded in time. Motivated by questions in passive scalar turbulence, a simple candidate for an exponentially mixing flow which is both bounded and spatially smooth was introduced and studied numerically by Pierrehumbert in the classical work \cite{Pierrehumbert94}. Pierrehumbert's model, which consists of alternating sine shears with random phases, was proven to be an exponential mixer just recently \cite{BGCZ22}. The general random dynamics framework developed in \cite{BGCZ22} has since been applied to alternating sine shears with different randomization procedures \cite{Cooperman22}. 

\subsubsection{Comparison with existing works}

In two spatial dimensions, there cannot exist a dissipation enhancing
flow with autonomous velocity field on the torus \cite{CKRZ08}. This
is because, under suitable conditions, divergence-free vector fields
on $\mathbb{T}^2$ admit a periodic Hamiltonian, which would be
preserved under the flow. Because of this, the velocity field must be
sufficiently complicated for the flow to be dissipation enhancing,
especially so if we want the flow to enjoy sharp decay rates. A
natural step after looking for a flow with autonomous velocity field
is to look for one with a time-periodic velocity field. Here, the
restrictions imposed by the Hamiltonian structure are relaxed, though
not entirely removed. It follows from our construction that the simple
(time-periodic) alternation of two autonomous and Lipschitz continuous
velocity fields can lead to optimal dissipation enhancement (and
mixing). As remarked earlier, the only previous example of dissipation
enhancement on the $|\log\nu|$ timescale that we know of are the flows
constructed as solutions to stochastic PDEs \cite{BBPSdiss}, which
while spatially regular are neither time-periodic nor uniformly
bounded. It appears that Theorem~\ref{thrm:enhanced} provides the
first example in the setting of time-periodic flows and also in the
setting of uniformly Lipschitz flows. An instance of convergence on
the $|\log \nu|$ timescale in a discrete-time setting has very
recently appeared for one-sided Bernoulli shifts on the torus with
axis-aligned cylinder sets perturbed by small noise
\cite{IyerNolanLu23} (here, $\nu$ is the parameter that controls the
strength of the noise, which plays a role analogous to the Laplacian
in \eqref{eq:advectiondiffusion}). Such maps however cannot be
realized as the time-1 flow map of a velocity field on
$\T^2$. Additionally the analysis \cite{IyerNolanLu23} seems not to
apply to invertible dynamics. We mention that the viewpoint of analyzing enhanced dissipation in terms of the convergence rate for an associated Markov process, which we utilize in our work, is present also in \cite{IyerNolanLu23} as well as \cite{IyerZhou22}.

In regards to exponential mixing, our construction can be contrasted
to the examples in \cite{Cooperman22, BGCZ22} where the velocity has
\emph{analytic} regularity in space, though the velocity is not
time-periodic and is based on \emph{random} switching in time while
ours is deterministic and periodic. The construction can also be contrasted with the previous work \cite{ElgindiZlatos19}, where less regular exponential mixers were constructed by finding a velocity field whose time-1 map is the Baker's transformation. The velocity field considered in the present paper was analyzed numerically in \cite{ZlatosNumerical} and observed to be exponentially mixing when $\alpha$ is not too small.  Theorem~\ref{thrm:mix} above establishes exponential mixing, at least if $\alpha$ is large enough. The fact that exponential mixing \textit{does not} occur for small $\alpha$ was proven in the very recent paper \cite{Hill23}.
Most similar to our setting is the recent work \cite{Hill22}, in which
the general theorems from \cite{ChernovZhang05, KatokStrelcyn} are
used to obtain exponential mixing for a class of non-monotone,
piecewise linear alternating shear flows without the dissipation ($\nu=0$). The
shears have fixed amplitude and are taken sufficiently asymmetric to
ensure good mixing properties of the map, whereas in our example the
shears are symmetric and uniform hyperbolicity is guaranteed by taking
the amplitude large. We provide a short and self-contained proof of the
mixing described by Lemma~\ref{lem:weakmix}, which is particularly well suited for
proving our principle result concerning the dissipation enhancement as $\nu
\rightarrow 0$.

\subsection{Discussion of the proof}
The proof of the Theorem~\ref{thrm:enhanced} is based primarily on considering the ``pulsed'' diffusion
obtained by applying the dissipation discretely in time and a careful study of the chaotic properties of the dynamical system generated by $u_\alpha$. We then
deduce enhanced dissipation in continuous time by using
an approximation argument that compares the solution to
\eqref{eq:advectiondiffusion} with the pulsed diffusion.

It is
important to not think of this scheme as treating the $\nu>0$ dynamics as
a perturbation of the $\nu=0$ dynamics even though we are interested
in the $\nu \rightarrow 0$ limit. The effect of the noise is critical
to our results as it produces the small scale mixing. Equally important is the fact that the structures which produce the
large scale mixing at $\nu=0$ persist when the system is viewed
correctly when $\nu >0$. As we will study the pulsed diffusion using its probabilistic representation, it is best to think of the pulsed diffusion as a time splitting scheme for the stochastic differential equation (SDE) whose probability flow produces \eqref{eq:advectiondiffusion}.


We now discuss the main steps in our approach, which hinges on a few
basic ideas from probability, hyperbolic dynamics, and PDE.

\subsubsection{Probabilistic setup for the pulsed diffusion}\label{sec:ProbSetup}

The dynamics given by \eqref{eq:advectiondiffusion} and
\eqref{eq:advection} describe the spread of an initial particle
distribution $f_0$ transported by the velocity $u$ with and without diffusion. The dynamics of an
individual particle starting at $z$ without diffusion is given by
$t \mapsto \phi_t(z)$, where $\phi_t$ denotes the flow map defined in \eqref{eq:flowmap}. In the presence of diffusion,  the dynamics of an
individual particle is given by the 
SDE
\begin{align}\label{eq:SDE}
 \varphi_{t,W} (z) =z + \int_0^t u(s, \varphi_{s,W} (z))ds +\sqrt{2\nu} W_t,
\end{align}
where $W_t$ is a standard two-dimensional Brownian motion $[0,\infty)$ with $W_0 = 0$.  Then
 when $\nu >0$, 
 $f(t,z)=\EE (\,f_0 \circ\varphi^{-1}_{t,W}(z)\,)$ where  $\EE$ is the expectation over the Brownian paths.


Here and throughout the remainder of the paper we write $\phi_t$ for the flow map associated with $u_\alpha$, suppressing the dependence on $\alpha$ for simplicity. The starting point for our analysis is to consider the pulsed
diffusion
\begin{align}\label{eq:phi}
  \Phi_{\nu} g = e^{\nu \Delta}(g \circ T),
\end{align}
where $T = \phi_1^{-1}$ and $e^{\nu \Delta}$ denotes convolution with
the periodized heat kernel. Then, $\Phi_{\nu} f_0$ is an approximation
for $\EE(\, f_0\circ \varphi^{-1}_{t,W}(z)\,)$. To parallel the
pathwise/particle level description given in \eqref{eq:flowmap} and
\eqref{eq:SDE} for the pulsed diffusion, we introduce the ``kicked'' version of $T$ by $ T_{v}(z) =
T(z + \pi(v))$ for $v \in \R^2$, where $\pi:\R^2 \to \T^2$ denotes the
projection $\pi(x,y) = (x \mod 1, y \mod 1)$. Observe that if $\xi$ is
a normal random variable on $\R^2$ with mean zero and identity  covariance,
then
\begin{equation}\label{eq:phi2}
  \Phi_{\nu} g(z) = \EE_\xi(\, g\circ T_{\sqrt{2\nu} \xi}(z) \, ):=
  \int_\R g(  T_{\sqrt{2\nu} v}(z))\, \mathcal{N}(dv)
\end{equation}
for $z \in \T^2$, where $\mathcal{N}$ denotes the standard unit normal distribution on $\R^2$.

To consider iterates of $\Phi_\nu$, we introduce the following
probabilistic framework. Define the probability space $(\Omega,\mathcal{F},\P) = (\R^2, \mathcal{B}(\R^2), \mathcal{N})^\N$. Then, for $\xi = (\xi_1,\xi_2,\ldots) \in \Omega$, define the ``randomly kicked'' map 
$$T_{\xi}^n = T_{\xi_{n}} \circ \ldots \circ T_{\xi_1}.$$
For
notational convenience, we also set $T_{\nu,\xi}^n = T_{\sqrt{2\nu} \xi}^n$. Given any $z \in \T^2$, $\{T_{\nu,\xi}^n(z)\}_{n=0}^\infty$ defines a Markov chain on $\T^2$ with initial condition $z$. 
As before the connection between $T_{\nu,\xi}^n$ and $\Phi_{\nu}$ is
given by 
	\begin{equation} \label{eq:probrep}
	\Phi_{\nu}^n g(z) = \int_{\Omega} g(T^n_{\nu,\xi}(z))\P(\dee \xi). 
	\end{equation}

	\subsubsection{Uniform hyperbolicity and the cone condition}
        \label{sec:UniformHyperbolicity}%
A fundamental property of the $\nu = 0$ dynamics is the exponential separation of particle trajectories both forward and backward in time. Indeed, taking $t \in \mathbb{N}$ without loss of generality, it is not hard to see that since $\grad V_\alpha$ and $\grad H_\alpha$ are both piecewise constant, for every $z \in \T^2$ away from a measure zero set of finitely many lines we have 
\begin{equation}
  \label{eq:GradiantAs}
  	\nabla T^{N}(z) 
          :=\nabla\phi_N^{-1}(z)=\Pi_{i=1}^{N}A_{j_{i}}, 
\end{equation}
        where $j_{i}\in\{1,2,3,4\}$ and $A_{j_i}$ is one of the four matrices 
	\[A_1=\begin{pmatrix} 1+\alpha^2 & \alpha \\ \alpha & 1 \end{pmatrix} ,\qquad A_2=\begin{pmatrix} 1-\alpha^2 & \alpha \\ -\alpha & 1 \end{pmatrix},\qquad A_3=\begin{pmatrix} 1-\alpha^2 & -\alpha \\ \alpha & 1 \end{pmatrix},\qquad A_4=\begin{pmatrix} 1+\alpha^2 & -\alpha \\ -\alpha & 1 \end{pmatrix}.\] 
 Of course, the particular product depends on the point $z$. There are four possible matrices since the vertical and horizontal shears are each linear on two halves of the torus. A key property that these matrices share is that for $\alpha$ large they are all hyperbolic with ``close" expanding and contracting directions. In particular, when $\alpha \gg 1$ the expanding direction of each is very close to $e_1=(1,0)$, while the contracting direction is very close to $e_2=(0,1).$ This can be made precise by establishing that the cones
	\[C_u:=\left\{(x,y)\in\mathbb{R}^2:|y|\leq \frac{|x|}{\alpha \delta_1}\right\},\]
	\[C_{s}:=\left\{(x,y)\in\mathbb{R}^2:|x|\leq \frac{|y|}{\alpha \delta_1}\right\}\] are, respectively, forward and backward invariant under the action of any of the $A_i$. Here, $\delta_1<1$ is fixed and $\alpha$ is sufficiently large. In particular, it can be shown using basic linear algebra that for any $v_u \in C_u$ and $v_s \in C_s$ we have 
	\[|\big(\Pi_{i=1}^N A_{j_i}\big)v_u|\geq \exp(cN),\qquad
          |\big(\Pi_{i=1}^N A_{j_i}\big)^{-1}v_s|\ge \exp(cN),\] for
        $\emph{any}$ choice of the sequence $\{j_i\}_{i=1}^N,$ with
        $c>0$ a constant growing with $\alpha$. Intuitively, the fact that each $A_i$ is hyperbolic guarantees a minimal amount of expansion in the unstable direction under each iteration, while the cone condition ensures that the expanding and contracting directions do not mix to cancel out the long-time effects. The observations above imply that the discrete time dynamical system on
        $\T^2$ defined by $T = \phi_1^{-1}$ is uniformly hyperbolic for
        $\alpha$ large. A critical fact in all that follows is that
        the representation given in \eqref{eq:GradiantAs} also holds
        for $\nabla T^{N}_{\nu,\xi}(z)$, as the additive shifts only
        rigidly translate the solutions and do not change the
        gradients. 
	
	\subsubsection{Complexity of intersections of approximate
          stable and unstable foliations}\label{sec:complexityintro}
	
	The cone conditions and uniform hyperbolicity established
above allow us to infer exponentially fast gradient growth as
$t\rightarrow\infty$ for generic Lipschitz solutions to
\eqref{eq:advection}. 
This itself does not imply mixing or dissipation enhancement (at the sharp rate for the latter). For that, we need to take more advantage of the above structure to get geometric information on the evolution of areas. In view of \eqref{eq:rectanglemix}, we would like to show that squares of side length $2^{-N}$ are stretched horizontally across the whole torus in time on the order of $N$ under the \emph{forward} iterates of $T$ and similarly stretched vertically under the \emph{backward} iterates. If this were the case, we would be able to argue that 
\begin{align}
  \label{eq:recMix1}
  	m(T^{2n}(R)\cap Q) = m(T^n(R) \cap T^{-n}(Q)) \gtrsim 
m(R)\,m(Q),
\end{align}
for $n \approx N$ and any squares $R$ and $Q$ of side length $2^{-N}.$ Given the forward invariance of $C_u$ and the fact that $T$ maps any line segment to a union of line segments, it is natural to first study the expansion of short horizontal lines. The challenge here is that while a line segment $W$ tangent to the unstable cone is expanded forwards in time, it is also repeatedly cut across the singularity curves of $\grad T$ (see Figure~\ref{fig:Sing} for a representative singularity set). Thus, there is a competition between expansion and cutting, making it not immediately clear the extent to which ``long lines dominate'' in $T^n(W)$. Nevertheless, using the cone conditions and the relatively simple structure of the singularity set of $\grad T$ we are able to prove that horizontal lines are stretched across the whole torus at the correct rate. Growth bounds of this general type on unstable curves are well known to be crucial in the chaotic dynamics literature and are often referred to as ``complexity estimates.'' We follow closely here the strategy of Baladi and Demers \cite{BaladiDemers20, Demers20}, with some modifications and simplifications that are available in our specific setting. 
 
 In order to make a somewhat precise statement, let us define $\mathcal{W}_u$ to be the collection of line segments on $\T^2$ that are tangent to some $v \in C_u$.
	
\begin{lemma} [Informal Statement] \label{lem:informalcomplexity}
		Let $W\in \mathcal{W}_u$. Then, if $n\gtrsim|\log |W||$, we have that $T^n(W)$ can be written as a union of line segments in $\mathcal{W}_u$ most of which have length greater than $5/4$. In particular, if $\alpha$ is sufficiently large, they stretch across $\mathbb{T}^2.$
	\end{lemma}
	
\noindent A similar statement holds backwards in time for line
segments tangent to the stable cone. The complexity lemma implies that horizontal segments forward in time
\emph{must} intersect vertical segments backward in time. From this,
we can show that squares that are mapped forward in time must
intersect squares that are mapped backwards in time with a significant
area of intersection.  To make this rigorous, we rely on an argument
using the area formula from geometric measure theory.

As in the preceding section, all of the arguments made here largely
carryover without modification to the randomly kicked version
$T_{\nu,\xi}^n$, as rigid translations of the solutions do not change
the estimates made.
	
\subsubsection{Area formula argument}
\label{sec:areaFormula}

The area formula allows us to reduce the estimate of $m(T^n(R) \cap T^{-n}(Q))$ to tracking the intersection points between horizontal lines mapped through $T^n$ and vertical lines mapped through $T^{-n}$. Define the Lipschitz continuous function $F_n:\T^2 \to \T^2$ by 
$$F_n(z) = (\Pi_x \circ T^n(z), \Pi_y \circ T^{-n}(z)),$$
where $\Pi_x$ and $\Pi_y$ denote the usual projections $\Pi_x(x,y) =
x$ and $\Pi_y(x,y) = y$. 
Since the unstable direction is approximately $e_1$ and the stable direction is approximately $e_2$, the first and second coordinates of $F_n(z)$ roughly give the forward and backward dynamics respectively in
the unstable direction. More precisely, $F_n^{-1}(x,y)$ consists exactly of the intersection points between $T^n(\ell_y)$ and $T^{-n}(\ell_x)$, where $\ell_y$ denotes the horizontal line with second coordinate $y$ that spans across the entire torus and $\ell_x$ similarly denotes the vertical line spanning the torus with first coordinate $x$. Using this fact, one can show by the area formula, recalled in Section~\ref{sec:area} for the convenience of the reader, that 
\begin{align*}
   m(T^{n}(R) \cap T^{-n}(Q)) = \int_{\Pi_x(Q)}\int_{\Pi_y(R)} \left(\sum_{z \in T^n(W_{y'}) \cap T^{-n}(V_{x'})} \frac{1}{J_{F_n}(z)}\right) \dee y' \dee x',
\end{align*}
  where $W_{y'} = \Pi_x(R) \times \{y'\}$, $V_{x'} = \{x'\} \times \Pi_y(Q)$, and $J_{F_n}$ denotes the Jacobian of $F_n$. Note that $W_{y'}$ and $V_{x'}$ are horizontal and vertical lines, respectively, of length $2^{-N}$. To estimate the sum for a fixed $(x',y')$, we take $n \gtrsim N$ and decompose $T^n(W_{y'})$ and $T^{-n}(V_{x'})$ each into the union of ``mostly long lines'' as guaranteed by the complexity lemma. The precise formulation of the statement that most elements of $T^n(W_{y'})$ and $T^{-n}(V_{x'})$ span the entire torus ensures that $\#(T^n(W_{y'}) \cap T^{-n}(V_{x'}))$ balances the Jacobian factor, yielding the bound  
  \begin{equation}\label{eq:lineintersect}
  \sum_{z \in T^n(W_{y'}) \cap T^{-n}(V_{x'})} \frac{1}{J_{F}(z)} \gtrsim |W_{y'}| |V_{x'}| =  2^{-2N}, 
  \end{equation}
  from which it follows immediately that 
  \begin{equation} \label{eq:rectangleequiv}  m(T^{2n}(R) \cap Q) = m(T^n(R) \cap T^{-n}(Q)) \gtrsim m(R) \,m(Q)
  \end{equation} 
  for $n \gtrsim N$. 
  
  The bound \eqref{eq:rectangleequiv} describes a kind of exponential mixing estimate. It extends to the randomly kicked maps in the form
  \begin{align*}
       m(T_\xi^{2n}(R) \cap Q) \gtrsim m(R) \,m(Q)
  \end{align*}
  for any $\xi \in \Omega$ and is the content of Lemma~\ref{lem:weakmix}, which we refer to as the geometric mixing lemma. The motivation for this name is that it follows as a corollary that the \textit{geometric mixing scale} (see Definition~\ref{def:geomixing}) of any solution to \eqref{eq:advection} with mean-zero $f_0 \in C^1$ decays exponentially fast. The geometric mixing scale has been considered in various works \cite{YaoZlatos17, ElgindiZlatos19, ACM19}, largely in connection with Bressan's conjecture. Exponential decay of the geometric mixing scale does not imply exponential mixing in the sense of Definition~\ref{ExponentialMixing} (see e.g. discussions in \cite{Zillinger19}), but we observe that Lemma~\ref{lem:weakmix} can similarly be used to show that $T^k$ is ergodic with respect to Lebesgue measure for every $k \in \N$, from which exponential decay of correlations follows by \cite[Theorem 2.8]{DemersLiverani08}. We do not require this fact for the proof of Theorem~\ref{thrm:enhanced}, which only relies on the geometric mixing lemma for the kicked maps.


\subsubsection{Enhanced dissipation}
\label{sec:EnhDisPulsedDiff}

Using \eqref{eq:probrep} and standard ideas from the ergodic theory of Markov processes, we show that the convergence $\Phi^n_\nu g \to \dashint g$ on the $|\log \nu|$ timescale (in $L^\infty$ for bounded $g$) follows provided that for every $\xi \in \Omega$ and rectangles $R,Q$ with side length $\sqrt{\nu}$ there holds 
\begin{equation}\label{eq:noisymix} 
  m(T^n_{\xi}(R) \cap Q) \gtrsim m(R)\, m(Q) 
\end{equation}
for $n \approx |\log \nu|$. This is exactly the geometric mixing lemma applied with squares at the diffusive $\sqrt{\nu}$ length scale.

 To extend the results for the pulsed diffusion to continuous time we use an approximation scheme based on a contradiction argument. In particular, we show that \emph{if the solution to \eqref{eq:advectiondiffusion}} decays slower than the expected rate, then the pulsed diffusion well-approximates the solution to \eqref{eq:advectiondiffusion}. The dissipation for solutions to the discrete problem are then transferred to solutions of \eqref{eq:advectiondiffusion}, which gives a contradiction. Compared with previous similar arguments in \cite{CZED20, FengIyer19}, here we are comparing two \emph{dissipative} problems rather than a dissipative problem with a conservative one. This gives us more control and in particular allows us to remove the exponent $2$ from the $|\log \nu|^2$ timescale often obtained in such arguments. 

\subsection{Outline of the article}

In Section~\ref{sec:mainproof} we prove Theorems~\ref{thrm:enhanced} and~\ref{thrm:mix} assuming the geometric exponential mixing estimate described above. This includes the necessary convergence results for the pulsed diffusion and the approximation argument that uses them to conclude optimal dissipation enhancement for $u_\alpha$. 
In Section~\ref{sec:dynamics} we prove the complexity lemma and in Section~\ref{sec:area} we conclude the proof of the geometric mixing lemma using the area formula argument.

	\section{Proof of Theorems~\ref{thrm:enhanced} and~\ref{thrm:mix}} \label{sec:mainproof}

 In this section, we first state precisely the geometric mixing lemma. Then, assuming it, we complete the proof of Theorems~\ref{thrm:enhanced} and~\ref{thrm:mix}. 

 \subsection{Geometric mixing lemma} \label{sec:prelimweakmixing}

For $N \in \Z_{\ge 0}$ we define a tiling of $\T^2$ with squares at scale $2^{-N}$ by
\begin{equation}
\mathcal{R}_N = \left\{[j 2^{-N},(j+1)2^{-N}]\times[k 2^{-N},(k+1)2^{-N}]: k,j \in \Z, 0 \le k,j \le 2^{N}-1 \right\}.
\end{equation}
The geometric exponential mixing lemma says that $m(T_\xi^n(R) \cap Q) \gtrsim m(R)\, m(Q)$ for any $R,Q \in \mathcal{R}_N$, $n \gtrsim N$, and $\xi \in (\R^2)^\N = \Omega$. In the statement below and all those that follow we assume that $\alpha$ is an even integer without explicit mention of it.

\begin{lemma}\label{lem:weakmix} 
There exist $C_0,\alpha_0 \ge 1$ and $\lambda \in (0,1)$ such that for $\alpha \ge \alpha_0$, $n \ge C_0N$, $\xi \in \Omega$, and any $R,Q \in \mathcal{R}_N$ there holds 
		\begin{equation} \label{eq:weakmix}
		m(T^n_{\xi}(R)\cap Q)\ge \lambda m(R)\,m(Q).
		\end{equation}
	\end{lemma}
	
The proof of Lemma~\ref{lem:weakmix} is the main technical step in the paper and the content of Sections~\ref{sec:dynamics} and~\ref{sec:area}.

\begin{remark}\label{rem:CouldGetUpper'}
     It is also possible to use the ideas of Sections~\ref{sec:dynamics} and~\ref{sec:area} to prove that the analogous upper bound $m(T_{\xi}^n(R) \cap Q) \le C m(R) \,m(Q)$ holds for a suitable constant $C$. The lower bound however is the fundamental estimate which describes mixing, and for the proof of Theorems~\ref{thrm:enhanced} and~\ref{thrm:mix} it is the only direction that we require.
 \end{remark}

\subsection{Probabilistic preliminaries} \label{sec:probability}

We now introduce some standard probabilistic notation that will be useful in the study of the pulsed diffusion and also recall a basic convergence theorem from the ergodic theory of Markov processes. 

Recall that for $\nu > 0$ we write $T_{\nu,\xi} = T_{\sqrt{2\nu}\xi}$
and that $T_{\nu,\xi}(z)$ defines a Markov chain for any initial
condition $z \in \T^2$. The associated Markov transition kernel is defined by 
$$\Phi_\nu(z,A) = \P\big(T_{\nu,\xi}(z) \in A\big),$$
where $A \in \mathcal{B}(\T^2)$, the Borel $\sigma$-algebra on
$\T^2$. Recall that a Markov transition kernel on a measurable space
$(X,\Sigma)$ is a mapping $Q:X \times \Sigma \to [0,1]$ such that
$Q(x,\cdot)$ is a probability measure for every $x \in X$ and $x
\mapsto Q(x,A)$ is measurable for every fixed $A \in \Sigma$. 
We similarly define $\Phi^n_\nu(z,A) = \P(T^n_{\nu,\xi}(z) \in A)$ and note that by the Markov property, for every
$n \ge 2$ we have
\begin{equation} \label{eq:Markov}
\Phi_\nu^n(z,A) = \int_{\T^2}\Phi_\nu^{n-1}(z',A)\,\Phi_\nu(z,\dee z')\,.
\end{equation}
This can be seen as simply the semigroup property $\Phi_\nu^{n+m} =
\Phi_{\nu}^{n}\Phi_\nu^m$ for the Markov semigroup $\Phi_\nu$ defined in
the introduction in \eqref{eq:phi} and \eqref{eq:phi2} since
$\Phi_\nu(z,A) =  \Phi_\nu\one_A(z)$, where the indicator function
$\one_A(z)$ is one if $z \in A$ and zero otherwise. It will also be convenient for us later on to define the shift operator $\theta:\Omega \to \Omega$, which acts on $\xi = (\xi_1,\xi_2,\ldots) \in \Omega$ by 
\begin{equation} \label{eq:shiftdef}
\theta \xi = (\xi_2,\xi_3,\ldots).
\end{equation}
	
To study the convergence of $\Phi_\nu^n g$ to $\dashint g$ we will make use of the following basic result from the ergodic theory of Markov processes. Recall that if $Q$ is a Markov transition kernel on a measurable space $(X,\Sigma)$, then a probability measure $\mu$ on $(X,\Sigma)$ is said to be \textit{invariant} for $Q$ if 
	\begin{equation} \label{eq:invariant}
	\mu(A) = \int_X Q(x,A)\,\mu(\dx)
	\end{equation}
	for every $A \in \Sigma$.
	
	\begin{lemma}[Doeblin] \label{thrm:Doeblin}
		Let $X$ be a Polish space and $\Sigma$ be the Borel $\sigma$-algebra on $X$. Let $Q$ be a Markov transition kernel on $(X,\Sigma)$ and define $Q^n$ as in \eqref{eq:Markov}. Suppose that there exists $\lambda \in (0,1)$ and a Borel probability measure $\mu$ on $(X,\Sigma)$ such that 
		\begin{equation} \label{eq:minorization} 
		Q(x,\cdot) \ge \lambda \mu(\cdot)
		\end{equation} 
		for every $x \in X$. Then, there exists a unique invariant measure $\mu_*$ for $Q$ such that for every bounded, Borel measurable function $f: X \to \R$ there holds 
		\begin{equation}
		\left\|\,Q^n f - \int_X f(x)\mu_*(\dx)\,\right\|_\infty \le 2(1-\lambda)^n\|f\|_\infty, 
		\end{equation}
		where $\|f\|_\infty = \sup_{x\in X}|f(x)|.$
	\end{lemma}
	
Since Lebesgue measure is invariant for $\Phi_\nu$, it follows from Lemma~\ref{thrm:Doeblin} that we can obtain an optimal convergence result for the pulsed diffusion by showing that $Q = \Phi_\nu^{n_0}$ satisfies \eqref{eq:minorization} for some $n_0 \approx |\log(\nu)|$ and $\lambda \in (0,1)$ that does not depend on $\nu$.
	
	\subsection{Convergence results for the pulsed diffusion} \label{sec:prelimpulsed}
	
In this section, we use Lemma~\ref{thrm:Doeblin} as described above to prove an optimal convergence estimate for the pulsed diffusion. The main result here is stated as follows. 

	\begin{theorem} \label{thrm:pulsed}
	For all $\alpha$ sufficiently large there exist constants $C, c > 0$ such that for any $\nu \in (0,1/2]$, $p \in \{2,\infty\}$, and mean-zero $f \in L^p(\T^2)$ there holds 
		\begin{equation} \label{eq:pulsed}
		\|\Phi_\nu^n f\|_{L^p} \le C e^{-c n|\log \nu|^{-1}}\|f\|_{L^p}.
		\end{equation}
	\end{theorem}

The key step in the proof of Theorem~\ref{thrm:pulsed} is to use Lemma~\ref{lem:weakmix} to obtain a uniform-in-$\nu$ lower bound on the density of $T_{\nu,\xi}^n(z_0)$ with respect to Lebesgue measure. For $n \ge 2$ we denote this density by $p_{n}(z_0,\cdot)$ and for $n = 1$ we write $p(z_0,\cdot)$. 

\begin{lemma} \label{lem:densitybound}
	Let $\alpha_0$, $C_0$, and $\lambda$ be as in the statement of Lemma~\ref{lem:weakmix}. There exist constants $\delta_1, C_1 > 0$ that do not depend on $\nu$ so that for all $\alpha \ge \alpha_0$, $z, z_0 \in \T^2$, and $n \ge C_1|\log\nu|$ we have
	\begin{equation} \label{eq:densitybound}
	 p_n(z_0,z) \ge \delta_1.
	\end{equation}
\end{lemma}

\begin{proof}
It suffices to consider the case where $\nu \ll 1$. Let $N$ be the first natural number such that $2^{-N} \le \sqrt{\nu}$ and $M \ge N$ be the first natural number such that $2^{-M} \le \nu^2$. Suppose that $n \ge C_0 M+2$ and note that this is satisfied provided $n \ge C |\log \nu|$ for some constant $C$ that does not depend on $\nu$. We first claim that for $\nu$ sufficiently small we have  
	\begin{equation} \label{eq:densityrectangle}
	 \inf_{R \in \mathcal{R}_N} \inf_{z_0 \in \T^2} \frac{\Phi_\nu^{n-1}(z_0,R)}{m(R)} \ge \frac{\lambda}{2},
	\end{equation}
 where $\lambda$ is as in Lemma~\ref{lem:weakmix}. By the semigroup property followed by Fubini's theorem, for any $Q \in \mathcal{R}_M$ we have 
	\begin{equation} \label{eq:fundbounds1}
	\Phi_\nu^{n-1}(z_0,Q) = \int_{\T^2} \Phi_\nu^{n-2}(z,Q) \,\Phi_\nu(z_0,\dz) = \int_{\Omega}\left(\int_{\T^2} \chi_{Q}(T_{\nu,\xi}^{n-2}(z))\,p(z_0,z)\,\dz \right) \P(\dee \xi),
	\end{equation}
 where $\chi_E$ denotes the characteristic function of a measurable set $E \subseteq \T^2$. Let $\bar{p}(z_0,\cdot)$ be the function that is constant on each $Q_j \in \mathcal{R}_M$ and such that 
	\begin{equation} \label{eq:fundbounds2}
	 \dashint_{Q_j} \bar{p}(z_0,z)\,\dz = \dashint_{Q_j}p(z_0,z)\,\dz:= c_j.
	 \end{equation}
	Since $p(z_0,z) = G(z_0 - T^{-1}(z))$, where $G$ is the fundamental solution to the heat equation on $\T^2$, we have $\|\grad p(z_0,\cdot)\|_{L^\infty} \le C_\alpha \nu^{-3/2}$ for some constant $C_\alpha$ depending on $\alpha$. It follows then that $\|\bar{p}(z_0,\cdot) - p(z_0,\cdot)\|_{L^\infty} \le C_\alpha\sqrt{2 \nu}.$ Now, by Lemma~\ref{lem:weakmix}, for every $\xi \in \Omega$ we have
	\begin{align}
	\int_{\T^2} \chi_{Q}(T_{\nu,\xi}^{n-2}(z))\,\bar{p}(z_0,z)\,\dz & = \sum_{Q_j \in \mathcal{R}_M} c_j \int_{\T^2} \chi_Q(T_{\nu,\xi}^{n-2}(z))\,\chi_{Q_j}(z)\,\dz \\ 
	& \ge \lambda\sum_{Q_j \in \mathcal{R}_M} c_j m(Q_j) \,m(Q) \\ 
	& = \lambda m(Q).
	\end{align}
	Putting this estimate into \eqref{eq:fundbounds1} and writing $p(z_0,z) = p(z_0,z) - \bar{p}(z_0,z) + \bar{p}(z_0,z)$ we obtain
 \begin{equation}
     \Phi_\nu^{n-1}(z_0,Q) \ge \lambda m(Q) - C_\alpha \sqrt{2\nu} \int_\Omega \left( \int_{\T^2} \chi_Q(T^{n-2}_{\nu,\xi}(z)) \dz\right) \P(\dee \xi) = (\lambda - C_\alpha \sqrt{2\nu})m(Q).
 \end{equation}
 Thus, for $\nu$ sufficiently small we have 
 \begin{equation} \label{eq:densityrectangle2}
 \inf_{Q \in \mathcal{R}_M} \inf_{z_0 \in \T^2} \frac{\Phi_\nu^{n-1}(z_0,Q)}{m(Q)} \ge \frac{\lambda}{2}. 
 \end{equation}
 The bound \eqref{eq:densityrectangle} then follows by writing $$\Phi_\nu^{n-1}(z_0,R) = \sum_{Q \in \mathcal{R}_M, Q \subseteq R} \Phi_\nu^{n-1}(z_0,Q)$$ and using \eqref{eq:densityrectangle2}.

We now use \eqref{eq:densityrectangle} to conclude the lower bound claimed in the lemma for $p_n(z_0,z)$. For ease of notation we write $h(z) = p_{n-1}(z_0,z)$. Note that $p_{n}(z_0,z) = (e^{\nu \Delta }h)(T^{-1}(z))$. Since composition with $T^{-1}$ preserves global lower bounds, it suffices to show that exists $\delta_1 > 0$ that does not depend on $\nu$ or $z_0$ such that 
	\begin{equation} \label{eq:fundbounds3}
	 (e^{\nu \Delta}h)(z) \ge \delta_1
	\end{equation}
	for all $z \in \T^2$. Let $\bar{h}$ denote the periodic extension of $h$ to $\R^2$ and define 
 $$P = [0,2^{-N+1}] \times [0,2^{-N+1}] \subseteq \R^2.$$
Observe that \eqref{eq:densityrectangle} implies 
	\begin{equation} \label{eq:fundbounds4}
	\int_{P} \bar{h}(z-z')\dee z' \ge \lambda 2^{-2N-1} \ge \frac{\lambda \nu}{8}
	\end{equation}
 for any $z \in \T^2$. Since $|z'|^2 \le 8 \nu$ for $z' \in P$ we thus have 
	\begin{align}
	(e^{\nu \Delta}h)(z) & = \frac{1}{4\pi \nu}\int_{\R^2} \exp\left(\frac{-|z'|^2}{4\nu}\right) \bar{h}(z-z')\dee z' \ge \frac{1}{4\pi \nu} \int_{P} \exp\left(\frac{-|z'|^2}{4\nu}\right) \bar{h}(z-z')\dee z' \ge \frac{\lambda}{32 \pi e^2}.
	\end{align}
\end{proof}

With Lemma~\ref{lem:densitybound} at hand we are now ready to prove Theorem~\ref{thrm:pulsed}. We will use Lemma~\ref{thrm:Doeblin} and the lower bound in \eqref{eq:densitybound} to deduce convergence in $L^\infty$ and then employ a lemma from \cite{IyerZhou22} to extend to the $L^2$ case.
	
	\begin{proof}[Proof of Theorem~\ref{thrm:pulsed}]
	Let $\delta_1, C_1 > 0$ be as in the statement of Lemma~\ref{lem:densitybound} and let $n_0$ be the first natural number greater than or equal to $C_1 |\log\nu|$. The lower bound in \eqref{eq:densitybound} implies that 
 \begin{equation}
\Phi_\nu^{n_0}(z,\cdot) \ge \delta_1 m(\cdot)
\end{equation}
for every $z \in \T^2$. Therefore, by Lemma~\ref{thrm:Doeblin} there exists a constant $\delta_2 > 0$ depending only $\delta_1$ such that for every $k \in \N$ and bounded, Borel measurable function $g$ with mean zero we have 
\begin{equation} \label{eq:Doeblin1}
\|\Phi_\nu^{k n_0} g\|_{\infty} \le 2 e^{-\delta_2 k} \|g\|_{\infty}.
\end{equation}
In other words, for every $n \in n_0 \N$ we have
\begin{equation} \label{eq:inttimes}
    \|\Phi_\nu^{n}g\|_{L^\infty} \le 2\exp\left(-\frac{\delta_2 n}{n_0}\right)\|g\|_{L^\infty} \le 2 \exp\left(-\frac{\delta_2 n}{(C_1 + 2)|\log \nu|}\right)\|g\|_{L^\infty}.
\end{equation}
This gives the desired estimate for $n \in n_0 \N$. The estimate at intermediate times follows in a standard way from \eqref{eq:inttimes} and the fact that $\Phi_\nu^m$ is bounded on $L^\infty$ uniformly in $m \in \N$.

To see how \eqref{eq:Doeblin1} implies the $L^2$ estimate, we follow
the proof of \cite[Lemma 4.1]{IyerZhou22}, whose argument we sketch for
completeness. For mean-zero $f \in L^2$, we have $\Phi_\nu^{n} f(z) = \int p_{n}(z,z')f(z') \dee z' =   \int [p_{n}(z,z')-1]f(z')
\dee z'$. Therefore, $ |\Phi_\nu^{n} f(z) |\leq\int\sqrt{
                   |p_{n}(z,z')-1|[p_{n}(z,z')+1]}|f(z')|\dee z'$ and hence the  Cauchy-Schwarz inequality implies that
\begin{align*}
  |\Phi_\nu^{n} f(z) |^2 \leq&\left(\int
                   |p_{n}(z,z')-1| \dee z' \right)\left(\int [p_{n}(z,z')+1]|f(z')|^2\dee z'\right)\,.
\end{align*}
Since Lebesgue measure is invariant for $\Phi_\nu^{n}$, we have
that $\int p_{n}(z,z')dz =1$ for all $z'$, and so $\int\int
[p_{n}(z,z')+1] |f(z')|^2\,\dee z\,\dee z'= 2\|f\|_{L^2}^2$. Combining these
last estimates gives
\begin{align*}
   \|\Phi_\nu^{n} f\|_{L^2}^2 \leq   2\|f\|_{L^2}^2 \sup_{z}\int
                   |p_{n}(z,z')-1| \dee z'. 
\end{align*}
Since by  \eqref{eq:Doeblin1} we have $\int
                   |p_{k n_0}(z,z')-1| dz' \leq   4 e^{-\delta_2 k} $, we
                   obtain
\begin{equation} \label{eq:Doeblin2}
    \|\Phi_\nu^{k n_0} f\|_{L^2} \le 2\sqrt{2} e^{-\frac12\delta_2 k}\|f\|_{L^2}.
  \end{equation}
The proof of the $L^2$ estimate is then completed from \eqref{eq:Doeblin2} in the same way as the $L^\infty$ estimate was from \eqref{eq:Doeblin1}.
\end{proof}
	
\subsection{Continuous-time enhanced dissipation} \label{sec:prelimcont}

In this section we prove that $u_\alpha$ is dissipation enhancing with the optimal rate function $\delta(\nu) = C|\log \nu|^{-1}$ by approximating the continuous-time diffusion with the pulsed diffusion and applying Theorem~\ref{thrm:pulsed}. The approximation relies only on the fact that the velocity field is Lipschitz and is thus quite general. We will therefore formulate it as an abstract result, as it may be of independent interest. 
	
	We begin by describing the general setting that we consider. Let $v :[0,1] \times \T^2 \to \R^2$ be a Lipschitz continuous (in space), divergence free velocity field satisfying 
 \begin{equation}
     \sup_{t \in [0,1]}\|v(t)\|_{\mathrm{Lip}} < \infty.
 \end{equation}
 We assume moreover that there is a partition $0 = t_0 < t_1 < \ldots < t_m = 1$ of $[0,1]$ so that $v \in C([t_{i-1},t_i)\times \T^2)$ for each $i = 1,\ldots, m$.  Let $\psi_t$ denote the flow map associated with $v$ and for $\nu>0$ let $\Psi_\nu$ be the operator given by 
	\begin{equation}
	\Psi_\nu f = e^{\nu \Delta}(f\circ \psi_1^{-1}).
	\end{equation}
The iterates of $\Psi_\nu$ define the pulsed diffusion generated by
$v$. 
	
	It is clear that $v=u_\alpha$ satisfies the assumptions above. Our general statement in the present setting is that the $L^2$ decay timescale of the advection-diffusion equation defined by the periodic-in-time extension of $v$ is at least as fast as the $L^2$ decay timescale of the pulsed diffusion $\Psi_\nu$. 
	
	\begin{lemma} \label{lem:approximation}
		Let $v$ and $\Psi_\nu$ be as above and for $f_0 \in L^2$ let $f$ denote the solution of \eqref{eq:advectiondiffusion} with $u$ the periodic-in-time extension of $v:[0,1)\times \T^2 \to \R^2$. Suppose that there exists $n_0 \in \N$ such that for every mean-zero $g \in L^2$ there holds 
		\begin{equation} \label{eq:discretegeneral} 
		\|\Psi^{n_0}_\nu g\|_{L^2} \le \frac{1}{2}\|g\|_{L^2}.
		\end{equation}
		Then, there exists $\delta \in (0,1)$ depending only on $ \sup_{0\le t \le 1}\|v(t)\|_{\mathrm{Lip}}$ such that for every mean-zero $f_0 \in L^2$ we have 
		\begin{equation} \label{eq:approxgoal}
		\|f(n_0+1)\|_{L^2} \le (1-\delta)\|f_0\|_{L^2}.
		\end{equation}
	\end{lemma}	
	
	By Theorem~\ref{thrm:pulsed} applied with $p = 2$, $\Phi_{\nu}$ satisfies \eqref{eq:discretegeneral} for some $n_0 \approx |\log(\nu)|$. The desired enhanced dissipation estimate for $u_\alpha$ at integer multiples of the time $t_0 = n_0$ then follows by iterating Lemma~\ref{lem:approximation}. The bound at the intermediate times is obtained using the monotonicity of $t \mapsto \|f(t)\|_{L^2}$ for $L^2$ solutions of \eqref{eq:advectiondiffusion}. To complete the proof of Theorem~\ref{thrm:enhanced}, it only remains to prove Lemma~\ref{lem:approximation}. 
	
Throughout the remainder of this section, for a given mean-zero $f_0 \in L^2$ we write $f(t)$ for the solution of \eqref{eq:advectiondiffusion} with $u$ the periodic-in-time extension of $v:[0,1) \times \T^2 \to \R^2$. In order to approximate the (continuous-time) $f$ by the (discrete-time) pulsed diffusion, it will be convenient to introduce some auxiliary continuous-time functions. First, let $F:[0,\infty) \times \T^2 \to \R$ be defined on the time interval $[n,n+1)$ for each integer $n \ge 0$ by
\begin{equation}
    F(t) = \begin{cases}
        f(n/2+(t-n)) & n \text{ even}, \\ 
        f((n+1)/2) & n \text{ odd}.
    \end{cases}
\end{equation}
That is, $F$ evolves in the same was as $f$ but is chosen to be constant on every other unit time interval. Next, we define the natural continuous-time operator $S_\nu^t: L^2 \to L^2$ satisfying $S_\nu^{2n} = \Psi_\nu^n$ for every $n \in \N$. In particular, for $g \in L^2$ let 
	\begin{equation} 
	S_\nu^t g = 
	\begin{cases}
	(\Psi_\nu^{\lfloor t/2 \rfloor} g)\circ \psi_{t-2\lfloor t/2 \rfloor}^{-1}& 0 \le t - 2\lfloor t/2 \rfloor \le 1, \\ 
	e^{\nu (t - 2\lfloor t/2 \rfloor - 1)\Delta}((\Psi_\nu^{\lfloor t/2 \rfloor} g)\circ \psi_1^{-1}) & 1 < t - 2\lfloor t/2 \rfloor < 2,
	\end{cases}
	\end{equation}
	where $\lfloor t/2 \rfloor$ denotes the largest integer less than or equal to $t/2$. Then, let $F_d(t) = S_\nu^t f_0$. This definition is simply such that $F_d$ solves the transport equation with velocity $v$ for $t \in (0,1)$, the heat equation for $t \in (1,2)$, and so on. 

 A key step in the proof of Lemma~\ref{lem:approximation} is an estimate on the error between $F_d$ and the continuous-time diffusion $f$. 
	
	\begin{lemma} \label{lem:L2approx}
		There exists a constant $C_1 > 0$ depending on $\sup_{0 \le t \le 1}\|v(t)\|_{\mathrm{Lip}}$ but not $\nu$ so that for every $n \in \N$, $\epsilon > 0$, and mean-zero $f_0 \in L^2$ we have
		\begin{equation} \|F_d(2n) - f(n)\|_{L^2}^2 = \| F_d(2n) - F(2n) \|_{L^2}^2 \le  C_1\epsilon \nu \int_0^{2n} \|\grad F_d(s)\|^2_{L^2} \dee s + C_1\epsilon^{-1}\nu \int_0^{n}  \|\grad f(s)\|^2_{L^2}\dee s.
		\end{equation}
	\end{lemma}
	
	\begin{proof}
 Here and throughout the remainder of this subsection, $C$ denotes a generic constant which may depend on $\sup_{0\le t \le 1}\|v(t)\|_{\mathrm{Lip}}$ but not on $\nu$ or the initial data $f_0$. We will prove that for every $n \in \N$ there holds 
 \begin{equation} \label{eq:approxgoal}
     \|F_d(2n) - F(2n)\|_{L^2}^2 \le C\epsilon \nu \int_0^{2n} \|\grad F_d(s)\|_{L^2}^2 \dee s + C \epsilon^{-1} \nu \int_0^{2n} \|\grad F(s)\|_{L^2}^2 \dee s.
 \end{equation}
 Given \eqref{eq:approxgoal}, the lemma follows because
 \begin{equation}
     \int_0^{2n} \|\grad F(s)\|_{L^2}^2 \dee s = \int_0^n \|\grad f(s)\|_{L^2}^2 \dee s + \sum_{k=1}^n \|\grad f(k)\|_{L^2}^2 \le C\int_0^n \|\grad f(s)\|_{L^2}^2 \dee s,
 \end{equation}
 where in the inequality we used that 
 $$ \|\grad f(k)\|_{L^2}^2 \le e^{Ct}\|\grad f(k-t)\|_{L^2}^2 \implies \|\grad f(k)\|_{L^2}^2 \le C\int_{k-1}^k \|\grad f(s)\|_{L^2}^2 \dee s, $$
 which holds by standard energy estimates for \eqref{eq:advectiondiffusion} with a uniformly-in-time Lipschitz velocity field. It remains then just to prove \eqref{eq:approxgoal}. We will assume $n = 1$, as the case $n > 1$ follows by iterating the same argument. For $t \in (0,1)$, we have 
		\begin{equation}
		\partial_t F_d + v \cdot \grad F_d = 0
		\end{equation}
		and 
		\begin{equation}
		\partial_t F + v \cdot \grad F = \nu \Delta F.
		\end{equation}
		Thus, an energy estimate gives
		\begin{equation}
		\frac{d}{dt}\|F_d(t) - F(t)\|_{L^2}^2 + 2\int_{\T^2} (F_d - F)v \cdot \grad (F_d - F) = 2\nu \int_{\T^2}(F - F_d) \Delta F .
		\end{equation}
		Integrating by parts we obtain
		\begin{equation}
		\frac{d}{dt}\|F_d(t) - F(t)\|_{L^2}^2 \le 2 \nu \|\grad F_d(t)\|_{L^2} \|\grad F(t)\|_{L^2},
		\end{equation}
and so
		\begin{equation} \label{eq:approx1}
		\|F_d(1) - F(1)\|_{L^2}^2 \le  2\nu \int_0^1 \|\grad F_d(t)\|_{L^2} \|\grad F(t)\|_{L^2} \dt.
		\end{equation}
		A similar computation using that for $t \in (1,2)$ we have 
$ \partial_t F_d  = \nu \Delta F_d$
  and $\partial_t F = 0$ shows that 
\begin{equation} \label{eq:approx2}
\|F_d(2) - F(2)\|_{L^2}^2 \le \|F_d(1) - F(1)\|_{L^2}^2 + 2\nu \int_1^2 \|\grad F_d(t)\|_{L^2} \|\grad F(t)\|_{L^2} \dt.
\end{equation}
		Combining \eqref{eq:approx1} and \eqref{eq:approx2} completes the proof of \eqref{eq:approxgoal} with $n = 1$.
	\end{proof}

 We now give a uniform-in-$\nu$ bound on $\nu \int_{0}^{\infty} \|\grad F_d(s)\|_{L^2}^2 \ds$ which we will use in the approximation estimate of Lemma~\ref{lem:L2approx}. 
	
	\begin{lemma}\label{lem:dissbound}
		Suppose that $f_0 \in L^2$ is such that 
		\begin{equation}
		\nu \|\grad f_0\|^2_{L^2} \le C_0 \|f_0\|_{L^2}^2
		\end{equation}
		for some constant $C_0 > 0$. Then, there exists a constant $C_2$ depending on $C_0$ and $\sup_{0 \le t \le 1}\|v(t)\|_{\mathrm{Lip}}$, but not on $\nu$ or $f_0$, 
	such that 
	$$\nu \int_0^\infty \|\grad F_d(t)\|_{L^2}^2 \dt \le C_2\|f_0\|_{L^2}^2.$$
	\end{lemma}
	
	\begin{proof}
		First observe that by the definition of $F_d$, the conservation of the $L^2$ norm for solutions of \eqref{eq:advection}, and the energy equality for the heat equation, we have
	\begin{equation} \label{eq:dissbound1}
		\nu \sum_{n=1}^\infty \int_{2n-1}^{2n}\|\grad F_d(t)\|_{L^2}^2 \dt \le \frac{1}{2}\|f_0\|_{L^2}^2.
			\end{equation} 
			Moreover, since $F_d$ solves the transport equation for $t \in (0,1)$ and $v$ is uniformly Lipschitz, we have
			\begin{equation} \label{eq:dissbound2}
			\nu \int_0^1 \|\grad F_d(t)\|_{L^2}^2 \dt \le C \nu \|\grad F_d(0)\|_{L^2}^2 = C \nu \|\grad f_0\|_{L^2}^2 \le CC_0 \|f_0\|_{L^2}^2.
			\end{equation}
   Combining \eqref{eq:dissbound1} and \eqref{eq:dissbound2}, it suffices to show that 
			\begin{equation} \label{eq:dissbound3}
			\nu \int_{2n-1}^{2n+1} \|\grad F_d(t)\|_{L^2}^2 \dt \le C \nu \int_{2n-1}^{2n} \|\grad F_d(t)\|_{L^2}^2 \dt
			\end{equation}
			for every $n \in \N$. We will just prove the estimate for $n = 1$, as the case where $n > 1$ is no different. For $t \in (1,2)$, the pulsed diffusion $F_d$ solves the heat equation, and hence $t \mapsto \| \grad F_d(t)\|_{L^2}$ is monotone decreasing. Thus, 
		\begin{equation} \label{eq:monotonediss} \nu \int_1^2 \|\grad F_d(t)\|_{L^2}^2 \dt \ge \nu \|\grad F_d(2)\|_{L^2}^2.
		\end{equation}
		Now, since $F_d$ solves the advection again for $t \in (2,3)$, we have
		\begin{equation} \label{eq:lipgradbound}
		\nu \int_2^3 \|\grad F_d(t)\|_{L^2}^2 \dt \le \nu \sup_{2 \le t \le 3}  \|\grad F_d(t)\|_{L^2}^2 \le C \nu \|\grad F_d(2)\|_{L^2}^2.
		\end{equation}
		Combining \eqref{eq:monotonediss} and \eqref{eq:lipgradbound} gives \eqref{eq:dissbound3} with $n = 1$, which completes the proof. 
  \end{proof}

  With Lemmas~\ref{lem:L2approx} and~\ref{lem:dissbound} at hand, we are ready to complete the proof of Lemma~\ref{lem:approximation}.
	
	\begin{proof}[Proof of Lemma~\ref{lem:approximation}]
		Fix mean-zero $f_0 \in L^2$ and assume for now that \begin{equation} \label{eq:regulardata}
  \nu\|\grad f_0\|_{L^2}^2 \le C_0\|f_0\|_{L^2}^2
  \end{equation}
  for some constant $C_0 > 0$. Let $n_0$ be as in the statement of Lemma~\ref{lem:approximation}. The basic energy estimate for $f$ on the time interval $[0,n_0]$ reads
		\begin{equation} \label{eq:basicenergy}
		\|f(n_0)\|_{L^2}^2 = \|f_0\|_{L^2}^2 - 2\nu \int_0^{n_0} \|\grad f(t)\|_{L^2}^2 \dee t.
		\end{equation}
  Suppose that 
  \begin{equation} \label{eq:contradiction} 2\nu \int_0^{n_0} \|\grad f(t)\|_{L^2}^2 dt \le \delta^2\|f_0\|_{L^2}^2
\end{equation}
for $\delta \in (0,1)$. Our goal is to prove that $\delta$ is bounded below independently of $\nu$. Let $F_d$ be as defined earlier and let $C_1$ and $C_2$ be as in the statements of Lemmas~\ref{lem:L2approx} and~\ref{lem:dissbound}, respectively. By \eqref{eq:contradiction}, Lemma~\ref{lem:L2approx} applied with $\epsilon = \delta$, and Lemma~\ref{lem:dissbound} we have 
		\begin{equation} \label{eq:finalapprox}
\|F_d(2n_0) - f(n_0)\|_{L^2}^2 \le C_1(C_2+1)\delta \|f_0\|_{L^2}^2:= C_3  \delta \|f_0\|_{L^2}^2.
		\end{equation}
		Using the assumption \eqref{eq:discretegeneral} we have 
		\begin{equation}
		\|F_d(2n_0)\|_{L^2} = \|\Psi_\nu^{n_0}f_0\|_{L^2} \le \frac{1}{2}\|f_0\|_{L^2}.
		\end{equation}
		Thus, from the triangle inequality and \eqref{eq:finalapprox} we get
		$$\|f(n_0)\|_{L^2}^2 \le 2 \|f(n_0) - F_d(2n_0)\|_{L^2}^2 + \frac{1}{2}\|f_0\|_{L^2}^2 \le \left(2\delta C_3 + \frac{1}{2}\right)\|f_0\|_{L^2}^2 ,$$
		which together with \eqref{eq:basicenergy} implies 
		$$2\nu \int_0^{n_0} \|\grad f(t)\|_{L^2}^2 \dt = \|f_0\|_{L^2}^2 - \|f(n_0)\|_{L^2}^2 \ge \left(\frac{1}{2} - 2\delta C_3\right)\|f_0\|_{L^2}^2.$$
		The previous estimate and \eqref{eq:contradiction} yield
		$$\frac{1}{2} - 2\delta C_3 \le \delta^2 \implies \delta \ge \frac{1}{2(1+2C_3)}. $$
	Combined with \eqref{eq:basicenergy}, we have shown that if \eqref{eq:regulardata} is satisfied, then there exists a constant $c > 0$ depending only on $C_0$ and $\sup_{0\le t \le 1}\|v(t)\|_{\mathrm{Lip}}$ such that 
 \begin{equation} \label{eq:approxstep1}
     \|f(n_0)\|_{L^2} \le (1-c)\|f_0\|_{L^2}.
 \end{equation}

We now turn to the general case where \eqref{eq:regulardata} does not necessarily hold. If
		\begin{equation} \label{eq:finalapprox1}
		\nu \int_0^1 \|\grad f(t)\|_{L^2}^2 \dt \ge (1/4)\|f_0\|_{L^2}^2,
		\end{equation} 
		then since 
		\begin{equation} \label{eq:finalapprox2}
		\|f(1)\|_{L^2}^2 = \|f_0\|_{L^2}^2 - 2\nu \int_0^1 \|\grad f(t)\|_{L^2}^2 \dt
		\end{equation}
		we see that $$\|f(t)\|_{L^2}^2 \le \|f(1)\|_{L^2}^2 \le (1/2)\|f_0\|_{L^2}^2$$ for every $t \ge 1$, and there is nothing to prove. We may thus assume that \eqref{eq:finalapprox1} is not satisfied, in which case Chebyshev's inequality and $\|\grad f(t)\|_{L^2} \le e^{C(t-s)}\|\grad f(s)\|_{L^2}$ imply 
		$$\nu \|\grad f(1)\|_{L^2} \le C\|f_0\|_{L^2}^2 \le C\|f(1)\|_{L^2}^2.$$
	In the last step above we used that 
$$
 \|f_0\|_{L^2}^2 \le 2\|f(1)\|_{L^2}^2$$
 whenever \eqref{eq:finalapprox1} does not hold. Thus, we may assume that $f(1)$ satisfies \eqref{eq:regulardata}. Applying \eqref{eq:approxstep1} with $f_0$ replaced by $f(1)$ we conclude 
 $$\|f(n_0 + 1)\|_{L^2} \le (1-c)\|f(1)\|_{L^2} \le (1-c)\|f_0\|_{L^2},$$
 which completes the proof.
	\end{proof}

 \begin{remark}
	As discussed in the introduction, similar arguments have appeared in previous works \cite{CZED20, FengIyer19} and been used to prove that a Lipschitz, exponentially mixing flow is dissipation enhancing on the timescale $|\log \nu|^2$. It is worth noting precisely what allows us to remove the 2 from the exponent of the logarithm. The standard approach is to approximate the advection-diffusion equation by the solution $f_0 \circ \psi_t^{-1}$ to the associated $\nu = 0$ problem and use a version of Lemma~\ref{lem:L2approx} with the error bound replaced by
	$$\|f_0 \circ \psi_t^{-1} - f(t)\|_{L^2}^2 \le 2\nu \int_0^t \|\Delta f(s)\|_{L^2} \|f_0 \circ \psi_s^{-1}\|_{L^2} \dee s.$$
Within this scheme, two derivatives are put on $f$ because there are no a priori bounds available on the $H^1$ norm of the approximating $\nu = 0$ solution. Estimating $\|\Delta f\|_{L^2}$, however, leads to important losses. We are able to put a derivative on our approximating solution $F_d$ because it experiences diffusion and hence satisfies good global $H^1$ bounds (Lemma~\ref{lem:dissbound} above). This is what underlies our ability to get the optimal enhanced dissipation timescale with our approximation method.
\end{remark}

\subsection{Exponential mixing for $u_\alpha$} \label{sec:prelimexpmixing}
	
In this section we obtain as a corollary of Lemma~\ref{lem:weakmix} that the geometric mixing scale of solutions to \eqref{eq:advection} decays exponentially fast. We also show how exponential mixing for $u_\alpha$ in the sense of Definition~\ref{ExponentialMixing} follows by applying in addition the main result of \cite{DemersLiverani08}. 

\subsubsection{Decay of the geometric mixing scale}

We begin by recalling the definition of the geometric mixing scale \cite{YaoZlatos17, ACM19, ElgindiZlatos19}.

\begin{definition} \label{def:geomixing}
Given fixed $\kappa \in (0,1)$, the geometric mixing scale of a mean-zero function $f \in L^\infty$ is defined as
\begin{equation} 
    \mathrm{mix}_\kappa(f) = \inf \left\{ 2^{-N}: \left|\;\dashint_{R} f\;\right| \le \kappa \|f\|_{L^\infty} \quad \forall R \in \mathcal{R}_N\right\}.
\end{equation}
\end{definition}

\noindent In the simple case where $f = \chi_A - \chi_{A^c}$, $\mathrm{mix}_\kappa(f) = 2^{-N}$ for $N$ the largest nonnegative integer such that every square $R \in \mathcal{R}_N$ consists of at least a fraction $(1-\kappa)/2$ of points where $f$ is both $+1$ and $-1$.

\begin{remark}
For $\kappa \in (0,1)$, the geometric mixing scale is typically defined as the infimum over all $\epsilon > 0$ such that 
\begin{equation} \label{eq:ballmix}
    \left| \;\dashint_{B_\epsilon(z)} f\; \right| \le \kappa \|f\|_{L^\infty}
\end{equation}
for every $z \in \T^2$, where $B_\epsilon(z)$ denotes the open ball of radius $\epsilon$ centered at $z$. Decay of $\mathrm{mix}_{\kappa'}(f)$ for some $\kappa' \in (0,1)$ in the sense of Definition~\ref{def:geomixing} implies decay of the geometric mixing scale defined using \eqref{eq:ballmix}. Indeed, it is easy to check that if $\mathrm{mix}_{\kappa'}(f) = 2^{-N}$, then there exists $\kappa \in (0,1)$ depending only on $\kappa'$ such that \eqref{eq:ballmix} holds for every $z \in \T^2$ and $\epsilon \le C2^{-N}$, where $C$ is a constant that does not depend on $\kappa'$ or $f$.
\end{remark}

\begin{theorem} \label{thrm:geomix}
For all $\alpha$ sufficiently large there exist constants $\kappa \in (0,1)$ and $C,c > 0$ such that for every mean-zero $f \in C^1$ we have 
\begin{equation}
    \mathrm{mix}_\kappa(f\circ \phi_t^{-1}) \le C \|f\|_{C^1}\|f\|_{L^\infty}^{-1}  e^{-ct}.
\end{equation}
\end{theorem}

\begin{proof}
Let $\mathcal{H}_N$ denote the set of simple functions which are mean zero and constant on each element of $\mathcal{R}_N$. We will first consider functions that are in some $\mathcal{R}_N$ and then extend to the general case by approximation. Fix $M \in \N$ and $f \in \mathcal{H}_M$. Let $C_0$ and $\lambda$ be as in Lemma~\ref{lem:weakmix}. Let $N \ge M$ and $n \ge C_0 N$. By Lemma~\ref{lem:weakmix}, for every $R_i, R_j \in \mathcal{R}_N$ we have 
\begin{equation}
    m(R_i \cap T^n(R_j)) = m(T^{-n}(R_i) \cap R_j) \ge \lambda m(R_i) m(R_j). 
\end{equation}
It follows that there exists a measurable set $Q_{ij} \subseteq R_i \cap T^n(R_j)$ such that $$m(Q_{ij}) = m(T^{-n}(Q_{ij}) \cap R_j) = \lambda m(R_i) m(R_j).$$ 
Note that the sets $Q_{ij}$ are disjoint because $T^n$ is a bijection. Define 
$$ P_i = R_i \setminus \bigcup_j Q_{ij}. $$
Since $f \in \mathcal{H}_M$ and $N \ge M$, we can write $f = \sum_{R_i \in \mathcal{R}_N} c_i \chi_{R_i}$ and decompose $f$ as 
$$f = \sum_{R_i \in \mathcal{R}_N} c_i \chi_{P_i} + \sum_{R_i, R_j \in \mathcal{R}_N} c_i \chi_{Q_{ij}}: = f_P + f_Q.$$
Then, for any $R_k \in \mathcal{R}_N$ we have 
\begin{equation} \label{eq:Qint}
    \dashint_{R_k} f_{Q} \circ T^n = \frac{1}{m(R_k)}\sum_{R_i \in \mathcal{R}_N} c_i m(T^{-n}(Q_{ik}) \cap R_k) = \lambda \sum_{R_i \in \mathcal{R}_N} c_i m(R_i) = \lambda \int_{\T^2} f = 0
\end{equation}
and
\begin{equation} \label{eq:Pint}
    \left| \dashint_{R_k} f_{P} \circ T^n\right| = \frac{1}{m(R_k)}\left|\sum_{R_i \in \mathcal{R}_N} c_i m(T^{-n}(P_i) \cap R_k)\right| \le \|f\|_{L^\infty} \frac{1}{m(R_k)}\sum_{R_i \in \mathcal{R}_N} m(T^{-n}(P_i) \cap R_k).
\end{equation}
Observe now that
\begin{align*} \sum_{i} m(T^{-n}(P_i) \cap R_k)
&= \sum_{R_i \in \mathcal{R}_N} m(T^{-n}(R_i) \cap R_k) - \sum_{R_i, R_j \in \mathcal{R}_N} m(T^{-n}(Q_{ij}) \cap R_k) \\
&= m(R_k) - \sum_{R_i, R_j \in \mathcal{R}_N}m(T^{-n}(Q_{ij}) \cap R_k) \\ 
& = m(R_k) - \sum_{R_i \in \mathcal{R}_N} m(T^{-n}(Q_{ik}) \cap R_k) \\
& = m(R_k) - \sum_{R_i \in \mathcal{R}_N} \lambda m(R_i)m(R_k) = (1-\lambda)m(R_k).
\end{align*}
Putting this equality into \eqref{eq:Pint} and applying also \eqref{eq:Qint} proves that 
\begin{equation} \label{eq:geosimple}
    \sup_{R \in \mathcal{R}_N}\left|\dashint_R f \circ T^n\right| \le (1-\lambda)\|f\|_{L^\infty}
\end{equation}
for every $N\ge M$ and $n \ge C_0 N$.

Now fix $f \in C^1$ and for $M$ to be chosen, let $\tilde{f} \in \mathcal{H}_M$ be such that 
\begin{equation}
   \tilde{f}(z) = \dashint_{R} f
\end{equation}
for every $R \in \mathcal{R}_M$ and $z \in R$. Then, by \eqref{eq:geosimple}, for every $N \ge M$, $n \ge C_0 N$, and $R \in \mathcal{R}_N$ we have 
\begin{equation}
    \left|\dashint_R f \circ T^n \right| \le (1-\lambda)\|f\|_{L^\infty} + 2^{-M+1}\|f\|_{\mathrm{Lip}}.
\end{equation}
Taking $M$ to the smallest natural number such that
\begin{equation} \label{eq:Mcriteria}
2^{-M+1} \le (\lambda/2)\|f\|_{L^\infty}\|f\|_{\mathrm{Lip}}^{-1},
\end{equation}
it follows that for $\kappa = 1-\lambda/2$ there holds
$$ \mathrm{mix}_{\kappa}(f \circ T^n) \le 2^{-N} $$
for every $N \ge M$ and $n \ge C_0 N$. This implies that for every $n \in \N$ we have 
\begin{equation} \label{eq:geomixdiscrete}
    \mathrm{mix}_\kappa (f \circ T^n) \le 2^{M}2^{-n/C_0} \le (2/\lambda)\|f\|_{\mathrm{Lip}}\|f\|_{L^\infty}^{-1} 2^{-n/C_0},
\end{equation}
which is the desired estimate at integer times. The estimate at general times follows by applying \eqref{eq:geomixdiscrete} with $f$ replaced by $f \circ \phi_s^{-1}$ for suitable $s \in (0,1)$ and noting that
$$ \|f\circ \phi_s^{-1}\|_{\mathrm{Lip}}\|f \circ \phi_s^{-1}\|_{L^\infty}^{-1} \le C\|f\|_{C^1}\|f\|_{L^\infty}^{-1}$$
for a constant $C$ depending only on $\alpha$.
\end{proof}

\subsubsection{Decay of correlations}

We now establish exponential mixing for $u_\alpha$ in the sense of Definition~\ref{ExponentialMixing}, completing the proof of Theorem~\ref{thrm:mix}. 

We begin with the exponential mixing of $T$. For $\alpha$ sufficiently large the map $T$ falls within the general class of piecewise hyperbolic maps studied in \cite{DemersLiverani08}. This is an easy consequence of the uniform hyperbolicity and singularity set structure of $T$ discussed in Section~\ref{sec:dynamics}. It follows from \cite[Theorem 2.8]{DemersLiverani08} that if $T^k$ is ergodic with respect to Lebesgue measure for every $k \in \N$, then $T$ is mixing and enjoys exponential decay of correlations for $C^1$ observables. Thus, we just need to prove that $T^k$ is ergodic, which we will show follows from Lemma~\ref{lem:weakmix}.

\begin{lemma} \label{thrm:Tmix}
    For all $\alpha$ sufficiently large, $T^k$ is ergodic with respect to Lebesgue measure for every $k \in \N$ and consequently there exist constants $c, C > 0$ such that 
    \begin{equation}
        \left|\;\dashint (f \circ T^n)\, g - \dashint f \;\dashint g\;\right| \le Ce^{-cn}\|f\|_{C^1}\|g\|_{C^1}
    \end{equation}
for every $f,g \in C^1(\T^2)$.
\end{lemma}

\begin{proof}
The proof that $T^k$ is ergodic relies only on the fact that Lemma~\ref{lem:weakmix} holds for $\xi = 0$ with $T$ replaced by $T^k$, and so we may take $k = 1$ without loss of generality. Let $\mathcal{H}_N$ be as defined in the proof of Theorem~\ref{thrm:geomix}. Using Lemma~\ref{lem:weakmix} and following the idea of the proof of \eqref{eq:geosimple}, we can show that there exists $\delta > 0$ such that for any $N \in \N$ and $f \in \mathcal{H}_N$ there exists $n \in \N$ for which
    \begin{equation} \label{eq:ergodicgoal}
    \left|\;\int_{\T^2} (f \circ T^n)\, f \;\right| \le (1-\delta)\|f\|_{L^2}^2.
    \end{equation}
    We omit the details for the sake of brevity.
    
    Let $f \in L^2$ be mean zero and invariant for $T$. That is, $f \circ T = f$ almost everywhere. We need to show that $f = 0$. Suppose for contradiction that this is not the case. Then, by standard approximation arguments, for any $\epsilon > 0$ there exists $N \in \N$ and $\psi \in \mathcal{H}_N$ such that $\|f - \psi\|_{L^2} \le \epsilon\|f\|_{L^2}$. Since $f$ is invariant for $T$, for any $n \in \N$ we get from the triangle inequality that
    \begin{equation}
        \|f\|_{L^2}^2 = \int_{\T^2}\; (f\circ T^n)\, f \; \le \left|\;\int_{\T^2} (\psi \circ T^n)\,\psi\; \right| + (3\epsilon^2 + 2\epsilon)\|f\|_{L^2}^2. 
    \end{equation}
    By \eqref{eq:ergodicgoal}, we can choose $n$ such that 
    $$ \|f\|_{L^2}^2 \le (1-\delta)(1+\epsilon)^2 \|f\|_{L^2}^2 + (3\epsilon^2 + 2\epsilon)\|f\|_{L^2}^2. $$
    Taking $\epsilon$ small enough so that $(1-\delta)(1+\epsilon)^2 + 3\epsilon^2 + 2\epsilon < 1$ yields a contradiction.

\end{proof}

It is now a fairly routine argument to upgrade to exponential mixing in continuous time.

 \begin{proof}[Proof of Theorem~\ref{thrm:mix}]
 Fix $f,g \in C^1(\T^2)$, which we may assume without loss of generality are both mean zero. Choose $t > 0$ and let $\lfloor t \rfloor$ denote the first integer less than or equal to $t$. By Lemma~\ref{thrm:Tmix}, there exist $C, c > 0$ such that 
     \begin{equation}
     \left| \;\int_{\T^2} (f\circ \phi_t^{-1})\,g \;\right| = \left|\; \int_{\T^2} (f \circ T^{\lfloor t \rfloor})\,(g \circ \phi_{t-\lfloor t \rfloor})\;\right| \le C e^{-c\lfloor t \rfloor} \|f\|_{C^1} \|g \circ \phi_{t-\lfloor t \rfloor}\|_{C^1}.
     \end{equation}
     Since $u_\alpha$ is uniformly Lipschitz, there exists a constant $C_\alpha > 0$ such that $\|g \circ \phi_{t-\lfloor t \rfloor}\|_{C^1} \le C_\alpha \|g\|_{C^1}.$ It follows that 
     \begin{equation}
     \left|\; \int_{\T^2} (f \circ \phi_t^{-1})\,g \;\right| \le (C C_\alpha e^{c}) e^{-c t} \|f\|_{C^1}\|g\|_{C^1},
     \end{equation}
     completing the proof. 
\end{proof}

	\section{Dynamics estimates} \label{sec:dynamics}

In this section, we obtain the main dynamics estimates necessary for the proof of Lemma~\ref{lem:weakmix}. In Section~\ref{sec:hyperbolic} we record precisely the basic facts concerning the uniform hyperbolicity of $T$. Then, in Section~\ref{sec:complexity} we prove the complexity lemma, which describes how the stretching of unstable curves under the iterates of $T$ dominates the cutting across singularity curves.
	
	\subsection{Uniform hyperbolicity} \label{sec:hyperbolic}

 Define the mappings $T_1$ and $T_2$ by
	\begin{equation} \label{eq:T1def}
	T_1(x,y) = 
	\begin{pmatrix}
	x \\
	y + \alpha |x-1/2|
	\end{pmatrix} \mod 1
	\quad \text{and} \quad 
	T_2(x,y) = 
	\begin{pmatrix}
	x+ \alpha|y-1/2| \\ y
	\end{pmatrix} \mod 1,
	\end{equation}
	and observe that $T_2 \circ T_1 =T$. Assuming $\alpha \ge 2$ is an even integer and writing $z = (x,y)$, it is easy to see that 
	\begin{equation} \label{eq:T1}
	T_1(z) = 
	\begin{cases}
	\begin{pmatrix}
	1 & 0 \\ 
	\alpha & 1
	\end{pmatrix} 
	z \mod 1 & x \ge 1/2 \\ 
	\begin{pmatrix} 1 & 0 \\ 
	-\alpha & 1 \end{pmatrix}z \mod 1 & x < 1/2,
	\end{cases}
	\end{equation} 
	and similarly
	\begin{equation}\label{eq:T2}
	T_2(z) = 
	\begin{cases}
	\begin{pmatrix}
	1 & \alpha \\ 
	0 & 1
	\end{pmatrix} 
	z \mod 1 & y \ge 1/2 \\ 
	\begin{pmatrix} 1 & -\alpha \\ 
	0 & 1 \end{pmatrix}z \mod 1 & y < 1/2.
	\end{cases}
	\end{equation} 
	Let 
	\begin{equation}
	\mathcal{A} = \left\{ \begin{pmatrix} 
	1 + \alpha^2 & \alpha \\ 
	\alpha & 1
	\end{pmatrix},
	\begin{pmatrix} 
	1 - \alpha^2 & \alpha \\ 
	-\alpha & 1
	\end{pmatrix},
	\begin{pmatrix} 
	1 - \alpha^2 & -\alpha \\ 
	\alpha & 1
	\end{pmatrix},
	\begin{pmatrix} 
	1 + \alpha^2 & -\alpha \\ 
	-\alpha & 1
	\end{pmatrix}
	\right\}
	\end{equation}
	denote the collection of possible products $A B$, where $A$ is one of the matrices in \eqref{eq:T1} and $B$ is one of the matrices in \eqref{eq:T2}. The map $T$ is smooth away from the singularity set 
	\begin{equation}\label{eq:singset}
	\mathcal{S}^+ = \left(\{x = 0\} \cup  \{x = 1/2\}\right) \cup T_1^{-1}(\{y = 1/2\} \cup \{y=0\}), 
	\end{equation}
	which partitions $\T^2\setminus \mathcal{S}^+$ into finitely many open, connected components such that on each one $Tz = Az \mod 1$ for some $A \in \mathcal{A}$. One can check by direct computation that there exists $C \ge 1$ such that for any $\alpha \ge 4$, every matrix $A \in \mathcal{A}$ has eigenvalues $\lambda_u = c_\alpha$, $\lambda_s = 1/c_\alpha$ for some $c_\alpha \in \R$ with $|c_\alpha| \ge \alpha^2/4$ and associated normalized eigenvectors 
	$$e_u = \begin{pmatrix} \cos \theta_u \\ \sin \theta_u\end{pmatrix}, \quad e_s = \begin{pmatrix} \cos \theta_s \\ \sin \theta_s \end{pmatrix}$$ 
	that satisfy $|\tan(\theta_u)| \le C\alpha^{-1}$ and $|\tan(\theta_s)| \ge C^{-1}\alpha$. As mentioned in the introduction, this implies that $T$ is uniformly hyperbolic for $\alpha$ large. We now describe this uniform hyperbolicity more precisely. Let $\mathcal{S}^{-} = T(\mathcal{S}^+)$ denote the singularity set for $T^{-1}$ and for $\delta_1 \in (0,1)$ let $C_u$ and $C_s$ be the cones defined in Section~\ref{sec:UniformHyperbolicity}. Then, for a suitable choice of $\delta_1$ and all $\alpha$ sufficiently large we have \[(\grad T)(z)C_{u} \subset C_{u}\qquad \text{and} \qquad (\grad T^{-1})(z)C_{s} \subset C_{s},\] where the two inclusions hold for all $z \in \T^2 \setminus \mathcal{S}^+$ 
	and all $z \in \T^2 \setminus \mathcal{S}^{-}$ respectively. Moreover, we have that  
	\begin{equation} \label{eq:hyperbolicity}
	\inf_{v \in C_{u}} \inf_{z \in \T^2\setminus \mathcal{S}^{+}} \frac{\|(\grad T)(z) v\|}{\|v\|} \ge \delta_1 \alpha^2, \quad \inf_{v \in C_{s}} \inf_{z \in \T^2\setminus \mathcal{S}^{-}} \frac{\|(\grad T^{-1})(z) v\|}{\|v\|} \ge \delta_1 \alpha^2.
	\end{equation}
	
	Since the uniform hyperbolicity of $T$ was a simple consequence of the stable and unstable eigenvectors being approximately aligned everywhere in space, it is clear that the hyperbolic properties of $T$ extend for free to the randomly kicked maps $T_\xi^n$. For $z_0 \in \R^2$, let $\mathcal{S}^+_{z_0} = \mathcal{S}^+ - \pi(z_0)$, where recall that $\pi:\R^2 \to \T^2$ is the natural projection. Then, for $\xi \in \Omega$ and $n \in \N$ define 
	\begin{equation}
	\mathcal{S}^{+,n}_\xi = \bigcup_{k=1}^{n} T_{\xi}^{-(k-1)}(\mathcal{S}^+_{\xi_k}), \quad \mathcal{S}_\xi^{-,n} = T_\xi^n(\mathcal{S}_\xi^{+,n}).
	\end{equation}
	With these definitions, $\mathcal{S}_\xi^{+,n}$ and $\mathcal{S}_\xi^{-,n}$ denote the singularity sets of $T_\xi^n$ and $T_{\xi}^{-n}$, respectively. The following lemma summarizes the hyperbolic properties of $T_\xi^n$ and $T_\xi^{-n}$ away from their singularity sets.
	
	\begin{lemma} \label{lem:hyperbolic}
		Let $\delta_1 \in (0,1)$ and $\alpha \gg 1$ be as above. Fix $\xi \in \Omega$ and $n \in \N$. For any $z \in \T^2\setminus \mathcal{S}_\xi^{+,n}$, $(\grad T^n_\xi)(z) \in \R^{2\times 2}$ is hyperbolic with normalized eigenvectors $e_u \in C_u$, $e_s \in C_s$ and associated eigenvalues $\lambda_u$, $\lambda_s$ which satisfy 
		$$|\lambda_u| \ge (\delta_1 \alpha^2)^n, \quad |\lambda_s| \le (\delta_1 \alpha^2)^{-n}.$$
		Moreover, there holds
		\begin{equation} \label{eq:hyperbolicityshifted}
		\inf_{v \in C_{u}} \inf_{z \in \T^2\setminus \mathcal{S}_\xi^{+,n}} \frac{\|( \grad T^n_\xi)(z) v\|}{\|v\|} \ge (\delta_1 \alpha^2)^n, \quad \inf_{v \in C_{s}} \inf_{z \in \T^2\setminus \mathcal{S}_\xi^{-,n}} \frac{\|( \grad T_\xi^{-n})(z) v\|}{\|v\|} \ge (\delta_1 \alpha^2)^n.
		\end{equation}
	\end{lemma}
	
	\begin{remark}
	A key aspect of the preceding lemma is that all of the relevant bounds are \emph{independent} of $\xi\in\Omega.$
	\end{remark}
	
	\subsection{Complexity lemma} \label{sec:complexity}
		
Recall from Section~\ref{sec:complexityintro} that $\mathcal{W}_u$ denotes the collection of line segments on $\T^2$ that are tangent to some $v \in C_u$. Since $T$ is piecewise linear, the strict invariance of the unstable cone implies that if $W \in \mathcal{W}_u$, then for any $\xi \in \Omega$ we can write $T_\xi(W)$ as a finite, disjoint union of elements in $\mathcal{W}_u$. The main result of this section is the complexity bound of Lemma~\ref{lem:informalcomplexity} from the introduction, restated precisely below. In what follows, $|\gamma|$ denotes the length of a piecewise smooth curve $\gamma$ on $\T^2$. 
	
	\begin{lemma} \label{lem:complexity}
		There is a constant $C \ge 1$ so that if $\alpha$ is sufficiently large, then for every $W \in \mathcal{W}_u$ with $|W| \le 1/2$, $\xi \in \Omega$, and $n \ge 1+ \log(1/|W|) + \log(8\alpha)$ there is a finite collection $\{W_i\}_{i \in I} \subseteq \mathcal{W}_u$ such that 
		$$ T_\xi^n(W) = \bigcup_{i \in I} W_i $$ 
		and
		\begin{equation}\label{eq:longdominate} \sum_{|W_i| \ge 5/4} |T_\xi^{-n}W_i| \ge \left(1-\frac{C}{\alpha}\right)|W|. 
		\end{equation}
		Moreover, $W_{i_1}$ and $W_{i_2}$ are disjoint when $i_1 \neq i_2$, and for each $i \in I$ the interior of $W_i$ with respect to the subspace topology of $T_\xi^{n}(W)$ is contained in $\T^2 \setminus \mathcal{S}_\xi^{-,n}$.
	\end{lemma}
	
	\begin{remark} \label{rem:longinterpret}
	The lower bound in \eqref{eq:longdominate} is the precise quantification of the fact that the line segments in $T_\xi^n(W)$ which span across the entire torus dominate for large $\alpha$. The complexity estimate in this form says that the fraction of the initial segment $W$ that is mapped to long lines is very close to one. Phrasing \eqref{eq:longdominate} in this way is crucial to the area formula argument of Section~\ref{sec:area}.
	\end{remark}
	
	\begin{remark} \label{rem:stableanalogue}
		Let $\mathcal{W}_s$ be defined in the same way as $\mathcal{W}_u$, but with $C_u$ replaced by the stable cone $C_s$. Then, a decomposition of $T^{-n}_\xi(W)$ for $W \in \mathcal{W}_s$ with $|W| \le 1/2$ exactly analogous to the one in Lemma~\ref{lem:complexity} holds.
	\end{remark}
	
As mentioned earlier, our proof of Lemma~\ref{lem:complexity} follows closely ideas from \cite{Demers20, BaladiDemers20}. We utilize the fact that $T$ is piecewise linear and the structure of the singularity set $\mathcal{S}^+$ to obtain the complexity estimate in the form of \eqref{eq:longdominate}, which is slightly different from the corresponding complexity lemmas in \cite{Demers20, BaladiDemers20}.  
	\subsubsection{Structure of the singularity sets}
We begin by noting that from \eqref{eq:singset}, we can deduce the following facts about $\mathcal{S}^+$, which is plotted in Figure \ref{fig:Sing}:
\begin{itemize}
    \item $\mathcal{S}^+$ is a union of $\{x=0\}, \{x=\frac{1}{2}\}$ and finitely many lines of slope $\pm \frac{1}{\alpha}$ that partition $\mathbb{T}^2\setminus\mathcal{S}^+$ into finitely many connected components $\mathcal{M}^+=\{M_i^+\}$. 
    \item The lines with slope $\frac{1}{\alpha}$ are contained in the region $0\leq x\leq \frac{1}{2}$ and are horizontally spaced by $\frac{1}{2\alpha}.$ The lines with slope $-\frac{1}{\alpha}$ are similarly in $\frac{1}{2}\leq x\leq 1.$
    \item The maximal number singularity curves that intersect at a point is three. 
    \item If $W \in \mathcal{W}_u$ with $|W| \le \alpha^{-1}/4$ and $\alpha$ is sufficiently large, then $W$ can have nonempty intersection with at most four distinct elements of $\mathcal{M}^+$.
\end{itemize}
Recall that for $z_0 \in \R^2$, $\mathcal{S}_{z_0}^+$ denotes the singularity of $T_{z_0}$. As the set $\mathcal{S}_{z_0}^+$ is obtained simply by translating $\mathcal{S}^+$, it has similar properties. In particular, the last bullet remains true for $\mathcal{S}^+_{z_0}$ for any $z_0 \in \R^2$ with $\mathcal{M}^+$ replaced by its translate $\mathcal{M}^+_{z_0}=\{M^+_{z_0,i}\}$.

	\begin{figure}[h]
		\centering
		\includegraphics[height=4in,width=7in]{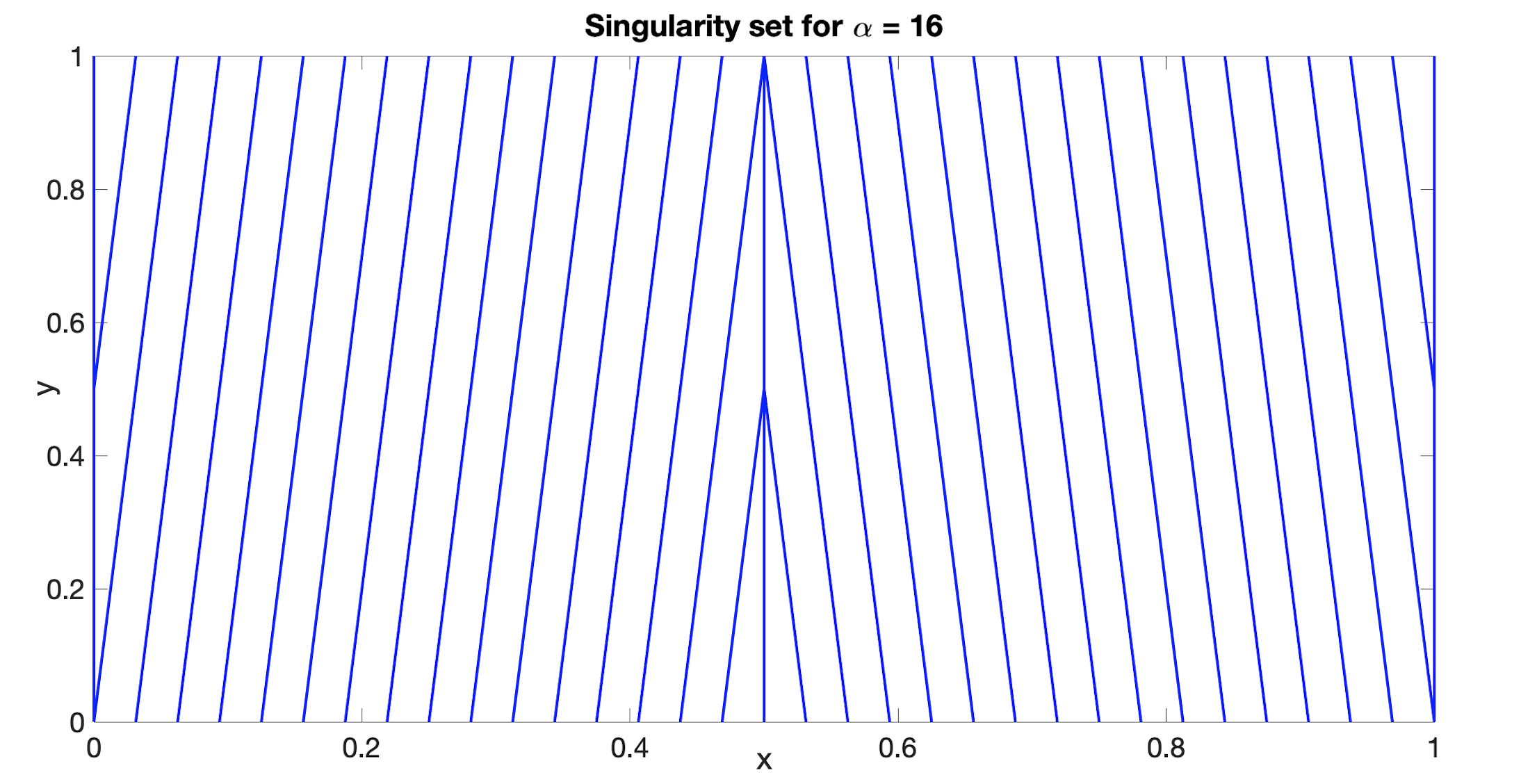}
		\vspace{-1 cm}
		\caption{Plot of the singularity set $\mathcal{S}^+$ on $\T^2$ in the case where $\alpha = 16$.}
		\label{fig:Sing}
	\end{figure}

\subsubsection{Generations of lines in $\mathcal{W}_u$}    	Given $\xi = (\xi_1,\xi_2,\ldots) \in \Omega$ and $W \in \mathcal{W}_u$ with $|W| \le \alpha^{-1}/4$, we define subsequent generations $\mathcal{G}_n(\xi,W)$ of lines in $\mathcal{W}_u$ with length less than or equal to $\alpha^{-1}/4$ obtained by mapping iteratively through $T_{\xi_1}$, $T_{\xi_2}$, and so on. First, let $\{M_{\xi_1,i_j}^+\}_{j \in J} \subseteq \mathcal{M}^+_{\xi_1}$ denote the distinct elements of $\mathcal{M}_{\xi_1}^+$ with which $W$ has nonempty intersection and note that as described above $\# J \le 4$. Let $\tilde{W}_j = T_{\xi_1}(W \cap \overline{M}_{\xi_1,i_j}^+)$ and then obtain $W_j \in \mathcal{W}_u$ from $\tilde{W}_j$ by modifying just the endpoints in such a way that $\{W_j\}_{j \in J} \subseteq \mathcal{W}_u$ consists of disjoint segments. We then define $\mathcal{G}(\xi_1,W) \subseteq \mathcal{W}_u$ as follows. If $|W_j| \le \alpha^{-1}/4$ then we declare $W_j \in \mathcal{G}(\xi_1,W)$. If $|W_j| > \alpha^{-1}/4$, then we subdivide $W_j$ into disjoint elements $\{W_{ij}\} \subseteq \mathcal{W}_{u}$ with $|W_{ij}| \in (\alpha^{-1}/8, \alpha^{-1}/4)$ and declare $W_{ij} \in \mathcal{G}(\xi_1,W)$ for each $i$. 
 Set $\mathcal{G}_1(\xi,W) = \mathcal{G}(\xi_1,W)$ and then for $n \ge 2$ define 
	\begin{equation} \label{eq:Gen1def}
	\mathcal{G}_n(\xi,W) = \bigcup_{V \in \mathcal{G}_{n-1}(\xi,W)} \mathcal{G}(\xi_n,V).
	\end{equation}
	It is clear that for every $n \in \N$ there holds
	\begin{equation}
	T_\xi^n(W) = \bigcup_{W_i \in \mathcal{G}_n(\xi,W)} W_i
	\end{equation}
	and that two line segments $W_{i_1}, W_{i_2} \in \mathcal{G}_n(\xi, W)$ are disjoint when $i_1 \neq i_2$. Moreover, note that by construction the interior of each $V \in \mathcal{G}_n(\xi, W)$ with respect to the subspace topology of $T_\xi^n(W)$ is contained in $\T^2 \setminus \mathcal{S}_\xi^{-,n}$. In other words, the inverse map $T^{-n}_\xi$ is smooth on a neighborhood of the interior of $V$. Similarly,  $T^{-k}_{\theta^{n-k}\xi}$ is smooth on a neighborhood of the interior of $V$ for every $k = 1, 2, \ldots, n$, where recall that $\theta:\Omega \to \Omega$ denotes the shift map.
	
	For $W \in \mathcal{W}_u$ and $\mathcal{G}_n(\xi,W)$ constructed as above, let $L_n(\xi, W)$ denote the elements $W_i \in \mathcal{G}_n(\xi,W)$ with $|W_i| \ge \alpha^{-1}/8$ and let $S_n(\xi, W) = \mathcal{G}_n(\xi,W) \setminus L_n(\xi, W)$. Here the notation is ``L'' for ``long'' and ``S'' for ``short.'' We define also $I_{n}(\xi, W)$ to be the elements $W_i \in S_n(\xi, W)$ such that for every $k = 1, 2, \ldots, n-1$ there exists $V_k \in S_{n-k}(\xi, W)$ such that $W_i \in S_k(\theta^{n-k}\xi, V_k)$. Stated roughly in simple terms, $I_n(\xi, W)$ is the collection of elements in $S_n(\xi, W)$
	that have always been contained in a short element. Lastly, for $\xi \in \Omega$, $k \in \N$, and a collection $E$ of disjoint elements of $\mathcal{W}_u$ we write 
	$$|E| = \sum_{W_i \in E}|W_i| \quad \text{and}\quad T^{\pm k}_\xi E = \{T_\xi^{\pm k}W:W\in E\}.$$ 

 \subsubsection{Proof of Lemma~\ref{lem:complexity}}
	The first step in the proof of Lemma~\ref{lem:complexity} is a lemma which says roughly that for an initial $W \in \mathcal{W}_u$ with $|W| \approx \alpha^{-1}$, the fraction of its initial length which corresponds to lines which never become long up to time $n$ decays exponentially in $n$.
	
	\begin{lemma} \label{lem:complexity1}
		Let $\epsilon \in (0,2)$. There exists $\alpha_0(\epsilon) \ge 0$ so that if $\alpha \ge \alpha_0$, then for every $\xi \in \Omega$, $W \in \mathcal{W}_u$ with $|W| \in(\alpha^{-1}/8,\alpha^{-1}/4]$, and $n \in \N$ there holds 
		\begin{equation}\label{eq:complexity1}
		|T_\xi^{-n} I_n(\xi, W)| \le \alpha^{-(2-\epsilon)n} |W|.
		\end{equation} 
	\end{lemma}
	
	\begin{proof}
		For any $|W| \le \alpha^{-1}/4$ we have $\# I_1(\xi,W) \le 4$. It is therefore easy to see by induction and the fact that \[I_{n}(\xi,W)=\bigcup_{V\in I_{n-1}(\xi,W)}I_{1}(\theta^{n-1}\xi,V)\] that 
		\begin{equation}  \label{eq:I_n1}
		\# I_n(\xi,W) \le 4^n
		\end{equation}
		for every $n \in \N$. On the other hand, by the hyperbolicity described in Lemma~\ref{lem:hyperbolic}, for any $V \in I_n(\xi,W) \subseteq S_n(\xi,W)$ we have that 
		\begin{equation} \label{eq:I_n2}
		|T_\xi^{-n}V| \le (\delta_1 \alpha^2)^{-n}|V| \le (\delta_1 \alpha^2)^{-n} (\alpha^{-1}/8) \le (\delta_1 \alpha^2)^{-n}|W|.
		\end{equation}
		Combining \eqref{eq:I_n1} and \eqref{eq:I_n2} yields
		$$\frac{|T_\xi^{-n}I_n(\xi, W)|}{|W|} \le \frac{4^n}{(\delta_1 \alpha^2)^n}.$$
	\end{proof}
	
Lemma~\ref{lem:complexity1} was a simple consequence of how the expansion rate  $\delta_1 \alpha^{2}$ dominates the number of singularity curves that a sufficiently small line segment can cross. We now use Lemma~\ref{lem:complexity1} applied at each iteration to obtain bounds on $|T_\xi^{-n} S_n(\xi, W)|$ and $|T_\xi^{-n} L_n(\xi, W)|$, which is ultimately what is required to prove Lemma~\ref{lem:complexity}. 
	\begin{lemma} \label{lem:complexity2}
		If $\alpha$ is sufficiently large, then for any $\xi \in \Omega$, $W \in \mathcal{W}_u$ with $|W| \le \alpha^{-1}/4$, and $n \ge  \log(1/|W|)$ there holds
		$$ |T_\xi^{-n}L_n(\xi,W)| \ge \left(1-\frac{1}{\alpha}\right)|W|.$$
	\end{lemma}
	
	\begin{proof}
		Fix $\xi \in \Omega$ and $W \in \mathcal{W}_u$. Since 
		$$|T_\xi^{-n}L_n(\xi,W)| + |T_\xi^{-n}S_n(\xi, W)| = |W|,  $$
		the lemma is equivalent to proving that

\begin{equation}\label{eq:shortratio}
		|T_\xi^{-n}S_n(\xi,W)| \le \frac{1}{\alpha} |W|.
		\end{equation}
		To estimate the left-hand side of \eqref{eq:shortratio}, following \cite{Demers20, BaladiDemers20}, we group the elements of $S_n(\xi,W)$ in terms of their most recent long ancestor. In particular, given any $W_i \in S_n(\xi, W)$, either $W_i \in I_n(\xi, W)$ or there is a unique $k \in \{1,\ldots,n-1\}$ such that $W_i \in I_{k}(\theta^{n-k}\xi, V)$ for some $V \in L_{n-k}(\xi, W)$. Thus, we can write 
		\begin{equation} \label{eq:complexity2}
		\frac{|T_\xi^{-n}S_n(\xi,W)|}{|W|} = \frac{|T_\xi^{-n}I_n(\xi,W)|}{|W|} + \frac{1}{|W|}\sum_{k=1}^{n-1} \sum_{V_i \in L_{n-k}(\xi,W)} |T_\xi^{-n}I_k(\theta^{n-k}\xi,V_i)|. 
		\end{equation}
		For the first term, observe that by the the proof of Lemma~\ref{lem:complexity} we have 
		$$\frac{|T_\xi^{-n}I_n(\xi,W)|}{|W|} \le \left(\frac{4}{\alpha^2 \delta_1}\right)^n \frac{1}{8\alpha |W|}.$$
		If $\alpha$ is large enough 
		then the choice of $n$ implies that 
		\begin{equation}
		\frac{|T_\xi^{-n}I_n(\xi,W)|}{|W|} \le \frac{1}{2\alpha}.
		\end{equation}
		For the second term in \eqref{eq:complexity2}, we first use the trivial fact that 
		$$|W| \ge \sum_{V_i \in L_{n-k}(\xi, W)} |T_\xi^{-(n-k)}V_i|$$
		for every $k$ to obtain 
		$$ \frac{1}{|W|}\sum_{k=1}^{n-1} \sum_{V_i \in L_{n-k}(\xi, W)} |T_\xi^{-n}I_k(\theta^{n-k}\xi,V_i)| \le \sum_{k=1}^{n-1} \frac{\sum_{V_i \in L_{n-k}(\xi,W)} |T_\xi^{-n}I_k(\theta^{n-k}\xi, V_i)|}{\sum_{V_i \in L_{n-k}(\xi, W)} |T_\xi^{-(n-k)}V_i|}.$$
		Using also the elementary inequality 
		$$ \frac{a_1 + a_2 + \ldots a_N}{b_1 + b_2 + \ldots b_N} \le \max_{1\le j \le N} \frac{a_j}{b_j} $$
		for two lists of positive numbers $\{a_i\}_{i=1}^N$ and $\{b_i\}_{i=1}^N$, we deduce
		\begin{equation}
		\frac{1}{|W|}\sum_{k=1}^{n-1} \sum_{V_i \in L_{n-k}(\xi, W)} |T_\xi^{-n}I_k(\theta^{n-k}\xi,V_i)| \le \sum_{k=1}^{n-1} \max_{V_i \in L_{n-k}(\xi, W)} \frac{|T_\xi^{-n}I_k(\theta^{n-k}\xi, V_i)|}{|T_\xi^{-(n-k)}V_i|}.
		\end{equation}
		We will now rewrite the sum above so that we are able to apply Lemma~\ref{lem:complexity1}. Fix $1 \le k \le n-1$ and $V_i \in L_{n-k}(\xi,W)$. Let $W_i \in I_k(\theta^{n-k}\xi,V_i)$ and observe that by the definition of $T^n_\xi$ and $\theta$ we have
		\begin{equation}
		T_\xi^{-n} W_i = T_\xi^{-(n-k)}(T^{-k}_{\theta^{n-k}\xi} W_i).
		\end{equation}
		Since $T_{\theta^{n-k}\xi}^{-k} W_i \subseteq V_i$ and $\grad T_\xi^{-(n-k)}$ is constant on $V_i$ it follows that 
		\begin{equation}
		\frac{|T_\xi^{-n} W_i|}{|T_\xi^{-(n-k)}V_i|} = \frac{|T_\xi^{-(n-k)}(T^{-k}_{\theta^{n-k}\xi} W_i)|}{|T_\xi^{-(n-k)}V_i|} = \frac{|T^{-k}_{\theta^{n-k}\xi}W_i|}{|V_i|}. 
		\end{equation}
		Summing this estimate over $W_i \in I_k(\theta^{n-k}\xi,V_i)$ and assuming that $\alpha$ is large enough so that Lemma~\ref{lem:complexity1} holds with $\epsilon = 1/2$, we obtain
		\begin{equation}
		\max_{V_i \in L_{n-k}(\xi,W)} \frac{|T_\xi^{-n}I_k(\theta^{n-k}\xi, V_i)|}{|T_\xi^{-(n-k)}V_i|} \le \max_{V_i \in L_{n-k}(\xi,W)}   \frac{|T_{\theta^{n-k}\xi}^{-k} I_k(\theta^{n-k}\xi,V_i)|}{|V_i|} \le \alpha^{-3(n-k)/2}.
		\end{equation}
		Combining our estimates thus far and assuming $\alpha$ is sufficiently large we have 
		\begin{equation}
		\frac{|T_\xi^{-n} S_n(\xi,W)|}{|W|} \le \frac{1}{2\alpha} + \sum_{k=1}^{n-1} \alpha^{-3(n-k)/2} \le \frac{1}{2\alpha}+ \frac{1}{\alpha^{3/2} - 1} \le \frac{1}{\alpha},
		\end{equation}
		which completes the proof. 
	\end{proof}
	
	We are now ready to give the proof of Lemma~\ref{lem:complexity}. Note that the difference between Lemma~\ref{lem:complexity} and Lemma~\ref{lem:complexity2} is that the ``long lines'' of Lemma \ref{lem:complexity} have length $5/4$ while those in Lemma~\ref{lem:complexity2} have length on the order of $\alpha^{-1}$. The remedy is just to modify the splitting given in the definition of the generations $\mathcal{G}_n$ at the last step. Indeed, the plan is to first apply Lemma~\ref{lem:complexity2} to obtain a set $\mathcal{G}_{n-1}(\xi,W)$ that consists primarily of lines with length approximately $\alpha^{-1}$, and then to consider the elements of $T_{\xi_n}(\mathcal{G}_{n-1}(\xi,W))$. The idea here is that since the expansion rate scales like $\alpha^2$ and $\alpha^2 \alpha^{-1} = \alpha \gg 1$, most line segments in $T_{\xi_n}(\mathcal{G}_{n-1}(\xi,W))$ necessarily span across $\T^2$. 
	
	\begin{proof}[Proof of Lemma~\ref{lem:complexity}]
		Fix $\xi \in \Omega$ and $W \in \mathcal{W}_u$. It is clear that by subdividing $W$ we may assume without loss of generality that $|W| \le \alpha^{-1}/4$.  
  For $z \in \R^2$ and $V \in \mathcal{W}_u$, let $\tilde{\mathcal{G}}(z,V)$ be defined in the same way as $\mathcal{G}(z,V)$ before but with the segments of length strictly greater than $5/2$ subdivided into segments with length strictly between $5/4$ and $5/2$. We define the collection $\tilde{\mathcal{G}}_n(\xi,W) \subseteq \mathcal{W}_u$ by 
		\begin{equation}
\tilde{\mathcal{G}}_n(\xi,W) = \bigcup_{W_i \in \mathcal{G}_{n-1}(\xi,W)} \tilde{\mathcal{G}}(\xi_n,W_i).
		\end{equation}
		Except for \eqref{eq:longdominate}, $\tilde{\mathcal{G}}_n(\xi,W)$ is easily seen by construction to satisfy all of the properties claimed in Lemma~\ref{lem:complexity}. It thus remains to prove, writing $\tilde{\mathcal{G}}_{n}(\xi,W) = \{W_i\}_{i \in I}$, that
		\begin{equation} \label{eq:longdominate2}
		\frac{1}{|W|}\sum_{|W_i| \ge 5/4}|T_\xi^{-n} W_i| \ge 1 - \frac{C}{\alpha}
		\end{equation}
		for some constant $C \ge 1$ that does not depend on $\alpha$ or $\xi$. First, since $n-1\geq \log \frac{1}{|W|}, $ by Lemma~\ref{lem:complexity2} we have 
		\begin{equation} 
		\frac{1}{|W|}\sum_{|W_i| \ge 5/4} |T_\xi^{-n}W_i| \ge \left(1-\frac{1}{\alpha}\right)\frac{1}{|T_\xi^{-(n-1)}L_{n-1}(\xi, W)|}\sum_{|W_i| \ge 5/4}|T_\xi^{-n}W_i|.
		\end{equation}
		To simplify notation, for $V_k \in L_{n-1}(\xi,W)$, define
		$$E_k = \{W_i \in \tilde{\mathcal{G}}_{n}(\xi,W): T_{\xi_n}^{-1} W_i \subseteq V_k, |W_i| \ge 5/4\}.$$ 
		Arguing similar to as in the proof of Lemma~\ref{lem:complexity2} we deduce 
		\begin{align*} \frac{1}{|T_\xi^{-(n-1)}L_{n-1}(\xi, W)|}\sum_{|W_i| \ge 5/4}|T_\xi^{-n}W_i| &\ge \frac{\sum_{V_k \in L_{n-1}(\xi,W)} \sum_{W_i \in E_k} |T_\xi^{-n}W_i|}{\sum_{V_k \in L_{n-1}(\xi,W)}|T_\xi^{-(n-1)}V_k|} \\ 
		& \ge \min_{V_k \in L_{n-1}(\xi,W)} \frac{\sum_{W_i \in E_k}|T_\xi^{-n}W_i|}{|T_\xi^{-(n-1)} V_k|}\\ 
		& = \min_{V_k \in L_{n-1}(\xi,W)} \frac{\sum_{W_i \in E_k}|T_{\xi_n}^{-1}W_i|}{|V_k|}.
		\end{align*}
		Using that $V_k \in L_{n-1}(\xi,W)$ can cross at most three singularity curves and the lower bound on the expansion rate given by Lemma~\ref{lem:hyperbolic},
		it is not hard to show that the construction of $\tilde{\mathcal{G}}(\xi_n,V_k)$ implies that 
		\begin{equation}
		\frac{\sum_{W_i \in E_k} |T_{\xi_n}^{-1}W_i|}{|V_k|} \ge \frac{|V_k| - \frac{15}{2\delta_1 \alpha^2}}{|V_k|} \ge 1 - \frac{60}{\delta_1 \alpha}.
		\end{equation}
		Combining the estimates above, we obtain 
		\begin{equation}
		\frac{1}{|W|} \sum_{|W_i|\ge 5/4} |T_\xi^{-n}W_i| \ge \left(1 - \frac{1}{\alpha}\right) \left(1 - \frac{60}{\delta_1 \alpha}\right),
		\end{equation}
		which completes the proof.
	\end{proof}
	
	\section{Proof of geometric mixing lemma} \label{sec:area}
	
	In this section we use the results of Section~\ref{sec:dynamics} to prove Lemma~\ref{lem:weakmix}. 
%
As discussed earlier, we make use of the area formula from geometric
               measure theory, which we recall below.
	\begin{lemma}[Area Formula] \label{lem:area}
		Let $M$ be an $m$-dimensional $C^1$ Riemannian manifold and
                let $\mathcal{H}^m$ be the m-dimensional Hausdorff
                measure on $M$ induced by the Riemannian metric. If
                $\mathcal{H}^m(M) < \infty$
                , $F:M \to M$ is a Lipschitz
               continuous map, and $\varphi: M \to \R$ is measurable, then 
               \begin{equation}
		\int_M J_F(z)\varphi(z)  \mathcal{H}^m(\dz) = \int_M \Bigg(\sum_{z \in F^{-1}(z')}\varphi(z)\Bigg)\mathcal{H}^m(\dee z'),
		\end{equation}
		where $J_F$ denotes the Jacobian of $F$ calculated
                using the Riemannian metric.
              \end{lemma}

              \begin{remark}
                The proof of the area formula can be found in 
                \cite[Theorem 3.2.3 or Theorem 3.2.5]{Federer} for an
                $m$-dimensional subset of $\R^n$. As discussed in
                \cite[Section 3.2.49]{Federer}, by localizing to
                charts which map $\R^m$ into $M$ we can extend the
                $\R^m$ result to the smooth manifold
                setting. Additional discussion can be found in
                \cite{Morgan, EvansGariepy, Brothers, Nicolaescu}.	
              \end{remark}

        \begin{remark} Since the function $F$ in the area formula is
          Lipschitz, Rademacher's Theorem (c.f. \cite[Theorem
          3.1.6]{Federer}) tells us that is differentiable almost
          everywhere. Thus, the Jacobian $J_F$ is defined almost
          everywhere by $J_F(z)=|\nabla F(z)|$, which is sufficient to
          define the integral in the area formula. If in addition
          $J_F(z)>0$ for almost every $z$, then
          \begin{equation}\label{eq:area2}
            	\int_M \varphi(z)  \mathcal{H}^m(\dz) = \int_M \Bigg(\sum_{z \in F^{-1}(z')}\frac{\varphi(z)}{J_F(z)} \Bigg)\mathcal{H}^m(\dee z').
			\end{equation}
This can be shown by subdividing $A = \{z \in M: J_F(z) > 0\}$ into the disjoint sets $A_n=\{z
               \in A : J_F(z) \in [\frac1{n+1}, \frac1{n} ) \cup
               [n,n+1)\}$ with $n\in\{1,2\dots\}$. On each $A_n$ the area
               formula implies \eqref{eq:area2} when applied to the
               function $\varphi/J_F$. Adding up the integrals over
               each $A_n$ implies the general result assuming the
               left-hand side of \eqref{eq:area2} is well defined. In
               our particular setting, $J_F$ will be uniformly bounded
               from above and below by positive constants, which makes
               \eqref{eq:area2} immediate from the area formula. 
        \end{remark}

 By making an appropriate choice of $F$, we apply Lemma~\ref{lem:area} to express $m(T_\xi^{2n}(R) \cap Q)$ as an integral that amounts essentially to counting intersection points (weighted by an appropriate Jacobian factor) between horizontal lines mapped through $T_\xi^n$ and vertical lines mapped through $T_{\theta^n \xi}^{-n}$.

	\begin{lemma} \label{lem:areaint}
		Fix $n,N \in \N$, $\xi \in \Omega$, and $Q, R \in \mathcal{R}_N$. Define the map $F_{n,\xi}:\T^2 \to \T^2$ by
		\begin{equation} \label{eq:defF}
		F_{n,\xi} = (\Pi_x \circ T^n_{\theta^n\xi},\Pi_y \circ T^{-n}_\xi),
		\end{equation}
		where $\Pi_x, \Pi_y:\T^2 \to \T$ are the projections $\Pi_x(x,y) = x$ and $\Pi_y(x,y) = y$. Let $W_{y,u} = \Pi_x(R) \times \{y\}$ and $W_{x,s} = \{x\} \times \Pi_y(Q)$. Then,
		$$ m(T_\xi^{2n}(R) \cap Q) = \int_{\Pi_x(Q)}\int_{\Pi_y(R)} \left(\sum_{z \in T_\xi^n(W_{y',u}) \cap T_{\theta^n \xi}^{-n}(W_{x',s})} \frac{1}{J_{F_{n,\xi}}(z)}\right) \dee y' \dee x'.
		$$
	\end{lemma}

\begin{proof}
Since $T$ is area preserving, we have 
	\begin{equation} \label{eq:T2toR2_1}
	m(T^{2n}_\xi(R) \cap Q) = m(T_\xi^n(R) \cap T_{\theta^n \xi}^{-n}(Q)) = \int_{\T^2} \chi_{R}( T_\xi^{-n}(z)) \chi_{Q}( T_{\theta^n \xi}^n(z)) \dz.
	\end{equation}
Thus, applying Lemma~\ref{lem:area} with $F_{n,\xi}$ as defined above we can write
	\begin{align}
	m(T^{2n}_\xi(R) \cap Q) & = \int_{\T^2} \left( \sum_{z \in F_{n,\xi}^{-1}(z')} \frac{\chi_R(T^{-n}_\xi(z)) \chi_Q(T_{\theta^n \xi}^n(z))}{J_{F_{n,\xi}}(z)}\right) \dee z' \\ 
 & = \int_{\Pi_x(Q)} \int_{\Pi_y(R)} \left(\sum_{z \in F_{n,\xi}^{-1}(x',y')} \frac{\chi_{\Pi_x(R)}(\Pi_x \circ T_\xi^{-n}(z)) \chi_{\Pi_y(Q)}(\Pi_y \circ T_{\theta^n \xi}(z))}{J_{F_{n,\xi}}(z)}\right) \dee y' \dee x'.
	\end{align}
 In the second line above we have observed that the definition of $F_{n,\xi}$ implies that if $z' = (x',y') \not \in \Pi_x(Q) \times \Pi_y(R)$, then at least one of $\chi_R(T_\xi^{-n}(z))$ or $\chi_Q(T_{\theta^n \xi}^n(z))$ vanishes for every $z \in F_{n,\xi}^{-1}(z')$. Moreover, for $(x',y') \in \Pi_x(Q) \times \Pi_y(R)$, the product of characteristic functions in the second line above is nonzero if and only if $z \in T_\xi^n (W_{y',u}) \cap T_{\theta^n \xi}^{-n}(W_{x',s})$. The lemma follows.
	\end{proof}

To estimate the sum in Lemma~\ref{lem:areaint} we use the decompositions of $T_\xi^n (W_{y',u})$ and $T_{\theta^n \xi}^{-n}(W_{x',s})$ guaranteed by Lemma~\ref{lem:complexity} and Remark~\ref{rem:stableanalogue}. The idea is essentially that for $(W,V) \in T_\xi^n (W_{y',u}) \times T_{\theta^n \xi}^{-n}(W_{x',s})$ with $|W|,|V| \ge 5/4$, we must have $\#(W \cap V) \ge 1$ for $\alpha$ large since the angle between vectors $v_u \in C_u$ and $v_s \in C_s$ converges to $\pm \pi/2$ as $\alpha \to \infty$. To properly employ \eqref{eq:longdominate} and obtain the correct quantitative estimate we use Lemma~\ref{lem:jacobian}, which relates $J_{F_{n,\xi}}(z)$ to the product of the arclength Jacobians along $W_{y',u}$ and $W_{x',s}$. 
	
	\begin{lemma} \label{lem:lbd}
		 For any $N
                 \in \N$ and $\alpha$ sufficiently large there exists $c > 0$ so that for all $\xi \in \Omega$, horizontal and vertical segments $(W_u, W_s) \in \mathcal{W}_u \times \mathcal{W}_s$ of length $2^{-N}$, and $n \ge 1+N\log(2) + \log(8\alpha)$ we have
		 \begin{equation} \label{eq:mainsumbound}
	\sum_{z
                  \in T_\xi^n(W_u) \cap T^{-n}_{\theta^n \xi}(W_s)}
                \frac{1}{J_{F_{n,\xi}}(z)} \ge c|W_u||W_s|,
		\end{equation}
  where $F_{n,\xi}$ is as before.
              \end{lemma}

\begin{remark}
    The upper bound mentioned in Remark~\ref{rem:CouldGetUpper'} is obtained by proving an upper bound analogous to \eqref{eq:mainsumbound}.
\end{remark}

	\begin{proof}
		 Let $G^u_n(W_{u}) \subseteq \mathcal{W}_u$ and $G^s_n(W_{s}) \subseteq \mathcal{W}_s$ denote the decompositions of $T_\xi^n (W_{u})$ and $T_{\theta^n \xi}^{-n}(W_{s})$ guaranteed by Lemma~\ref{lem:complexity} and Remark~\ref{rem:stableanalogue}. Define also $L_n(W_{u})$ as the elements $W \in G_n^u(W_{u})$ with $|W| \ge 5/4$ and similarly define $L_n(W_{s})$. Clearly, we have 
		 \begin{equation}
		 \sum_{z \in T_\xi^n(W_u) \cap T^{-n}_{\theta^n \xi}(W_s)} \frac{1}{J_{F_{n,\xi}}(z)} = \sum_{W \in G^n_u(W_u)} \sum_{V \in G^n_s(W_s)} \sum_{z \in W \cap V} \frac{1}{J_{F_{n,\xi}}(z)}.
		 \end{equation}
		 Let $J_{u,n}(z) = |\grad T_\xi^n(z) e_x|$ and $J_{s,n}(z) = |\grad T^{-n}_{\theta^n \xi}(z) e_y|$, where $e_x$ and $e_y$ are the standard basis vectors of $\R^2$. By Lemma~\ref{lem:jacobian} and trivially bounding the sum from below by excluding the short lines we have 
	
 \begin{align}&\sum_{W \in G_n^u(W_{u})} \sum_{V \in G_n^s(W_{s})} \sum_{z \in W \cap V} \frac{1}{J_{F_{n,\xi}}(z)} \\ 
		 & \quad \ge (1-C\alpha^{-1})\sum_{W \in L_n(W_{u})} \sum_{V \in L_n(W_{s})} \sum_{z \in W \cap V} \frac{1}{J_{u,n}(T_\xi^{-n}(z)) J_{s,n}(T^n_{\theta^n \xi}(z))} \\ 
		 & \quad = (1-C\alpha^{-1})\sum_{W \in L_n(W_{u})} \sum_{V \in L_n(W_{s})} \#(W \cap V) \times \frac{|T^{-n}_\xi W| |T_{\theta^n \xi}^n V|}{|W| |V|}.
		 \end{align}
		 For $W$ and $V$ as in the last sum above we have $\#(W \cap V) \ge 1$ and $|W|, |V| \le 5/2$, and hence by \eqref{eq:longdominate} we have 
		 \begin{align}
		 \sum_{W \in G_n^u(W_{u})} \sum_{V \in G_n^s(W_{s})} \sum_{z \in W \cap V} \frac{1}{J_{F_{n,\xi}}(z)} &\ge \frac{4}{20}(1-C \alpha^{-1}) \sum_{W \in L_n(W_u)} |T_\xi^{-n} W| \sum_{V \in L_n(W_s)}  |T_{\theta^n \xi}^n V| \\ 
		 & \ge \frac{4}{20} \left(1-C\alpha^{-1}\right)^3 |W_u||W_s|,
		 \end{align} 
		 which completes the proof.
	\end{proof}
	
	The proof of Lemma~\ref{lem:weakmix} now follows easily.
	
	\begin{proof}[Proof of
          Lemma~\ref{lem:weakmix}]
		Fix $N \in \N$. Let $\alpha$ be large enough so that Lemma~\ref{lem:lbd} applies and let $c > 0$ be as in \eqref{eq:mainsumbound}. Taking $n \ge 1 + N\log(2) + \log(8\alpha)$, it follows by Lemmas~\ref{lem:area} and~\ref{lem:lbd} that for any $R,Q \in \mathcal{R}_N$, and $\xi \in \Omega$ we have
\begin{align*}
    m(T_\xi^{2n}(R) \cap Q) \ge c \int_{\Pi_x(Q)}\int_{\Pi_y(R)} |W_{y',u}||W_{x',s}| \dee y' \dee x = c 2^{-4N} = cm(R)m(Q),
\end{align*}
which is the desired estimate.
\end{proof}

\appendix 
\section{Jacobian estimates}

We now prove the technical Jacobian estimate needed in the proof of Lemma~\ref{lem:weakmix}.

	\begin{lemma} \label{lem:jacobian}
	Fix $\xi \in \Omega$ and $n \in \N$. Define $J_{u,n}(z) = |\grad T_\xi^n(z) e_x|$ and $J_{s,n}(z) = |\grad T_{\theta^n \xi}^{-n}(z) e_y|$, and let $F_{n,\xi}$ be as defined in \eqref{eq:defF}. There exists a constant $C \ge 1$ that does not depend on $\xi$ or $n$ such that for all $\alpha$ sufficiently large there holds
	\begin{equation} \label{eq:jacobianlem}
	\left| \frac{J_{u,n}(T_\xi^{-n}(z))J_{s,n}(T^n_{\theta^n \xi}(z))}{J_{F_{n,\xi}}(z)} - 1\right| +  \left| \frac{J_{F_{n,\xi}}(z)}{J_{u,n}(T_\xi^{-n}(z))J_{s,n}(T^n_{\theta^n \xi}(z))} - 1\right| \le \frac{C}{\alpha}
	\end{equation}
	for every $z \in \T^2$ in the full measure set where the derivatives in \eqref{eq:jacobianlem} exist.
\end{lemma}

\begin{proof}
	By the elementary inequality 
	\begin{equation} \label{eq:elementary1}
 \left| \frac{a}{b} - 1\right| \le \frac{\left|\frac{b}{a} - 1\right|}{1- \left|\frac{b}{a}-1\right|}
 \end{equation}
	for nonzero $a,b \in \R$, it suffices to estimate only the first term on the left-hand side of \eqref{eq:jacobianlem}. We will also make use of the related inequality
	\begin{equation} \label{eq:elementary2}
	\left|\frac{c}{d}-1\right| \le \left|\frac{c}{a}-1\right|\left(1+\left|\frac{a}{d} - 1\right| \right) + \left|\frac{a}{d} -1\right|,
	\end{equation}
	which holds for any $a,b,c,d \in \R$ with $a$ and $d$ nonzero. In this proof, $C$ will denote a positive constant, which may change from line-to-line, that does not depend on $\xi$, $n$, or $\alpha$ for $\alpha$ sufficiently large.
	
	Let $z \in \T^2$ be such that all of the derivatives in \eqref{eq:jacobianlem} exist. We first claim that 
	\begin{equation} \label{eq:jacobianlem1}
	\frac{|\partial_y (\Pi_x \circ T_{\theta^n \xi}^n)(z)|}{|\partial_x (\Pi_x \circ T_{\theta^n \xi}^n)(z)|} \le \frac{C}{\alpha}
	\end{equation}
	for $\alpha$ sufficiently large. Indeed, let $e_s$ and $e_u$ denote the normalized stable and unstable eigenvectors for $\grad T_{\theta^n \xi}^n(z)$ and let $\lambda \in \R$ be the unstable eigenvalue. If $\theta \in [0,\pi)$ denotes the angle between $e_x$ and $e_s$, then since $e_s \in C_s$ we have $|\theta-\pi/2| \le C\alpha^{-1}$. Therefore, 
 $$ |e_x \cdot e_s| = |\cos\theta| = |\sin(\theta - \pi/2)| \le C\alpha^{-1}$$
 for $\alpha$ large. Using also that $e_u \in C_u$, we similarly have $$|e_y \cdot e_u| \le C\alpha^{-1}, \quad |e_x \cdot e_u| \ge 1-C^{-1}\alpha^{-1}, \quad \text{and} \quad |e_y \cdot e_s| \ge 1-C^{-1}\alpha^{-1}.$$ 
 For $\alpha$ large we can thus write 
	$$e_x = a_1 e_u + b_1 e_s \quad \text{ and } \quad e_y = a_2 e_u + b_2 e_s$$
	for coefficients satisfying $|a_1|, |b_2| \approx 1$ and $|a_2|, |b_1| \le C \alpha^{-1}$. Noting also that $|e_u \cdot e_s| \le C\alpha^{-1}$ we have
	\begin{align*}
	\frac{|\partial_y (\Pi_x \circ T_{\theta^n \xi}^n)(z)|}{|\partial_x (\Pi_x \circ T_{\theta^n \xi}^n)(z)|} & = \frac{|e_x \cdot \grad T_{\theta^n \xi}^n(z)e_y|}{|e_x \cdot \grad T_{\theta^n \xi}^n(z)e_x|} \\ 
	& = \frac{|a_1 a_2 \lambda +b_1 b_2 \lambda^{-1} + (e_u \cdot e_s) (a_1 b_2 \lambda^{-1} + a_2 b_1 \lambda)|}{|a_1^2 \lambda + b_1^2 \lambda^{-1} + (e_u \cdot e_s)a_1b_1(\lambda + \lambda^{-1})|} \\ 
	& \le \frac{C\alpha^{-1} |\lambda|}{a_1^2 |\lambda| - b_1^2 |\lambda|^{-1} - 2|\lambda||a_1 b_1| |e_u \cdot e_s|} \\ 
	& \le \frac{C \alpha^{-1} |\lambda|}{a_1^2|\lambda| - C\alpha^{-1}|\lambda|} \\ 
	& \le \frac{C}{\alpha}.
	\end{align*}
	This proves \eqref{eq:jacobianlem1}. The same argument shows that 
	\begin{equation} \label{eq:jacobianlem2}
	\frac{|\partial_x(\Pi_y \circ T_{\xi}^{-n})(z)|}{|\partial_y(\Pi_y \circ T_\xi^{-n})(z)|} \le \frac{C}{\alpha}, \quad \frac{|\partial_x(\Pi_y \circ T_{\theta^n \xi}^n)(z)|}{|\partial_x(\Pi_x\circ T_{\theta^n \xi}^n)(z)|} \le \frac{C}{\alpha}, \quad \text{and} \quad \frac{|\partial_y (\Pi_x \circ T_\xi^{-n})(z)|}{|\partial_y(\Pi_y \circ T_\xi^{-n})(z)|} \le \frac{C}{\alpha}.
	\end{equation}
	
	Combining \eqref{eq:jacobianlem1} and the first estimate in \eqref{eq:jacobianlem2} it is immediate from the formula for $J_{F_{n,\xi}}(z)$ that
	\begin{equation} \label{eq:jacobianlem3}
	\left|\frac{J_{F_{n,\xi}}(z)}{|\partial_x(\Pi_x\circ T_{\theta^n \xi}^n)(z)| |\partial_y(\Pi_y \circ T_\xi^{-n})(z)|} - 1 \right| \le \frac{C}{\alpha}.
	\end{equation}
	It is also straightforward to prove from the second two bounds in \eqref{eq:jacobianlem2} that 
	\begin{equation} \label{eq:jacobianlem4}
	\left|\frac{|\grad T_{\theta^n \xi}^n(z)e_x| | \grad T_{\xi}^{-n}(z)e_y|}{|\partial_x(\Pi_x\circ T_{\theta^n \xi}^n)(z)| |\partial_y(\Pi_y \circ T_\xi^{-n})(z)|} - 1 \right| \le \frac{C}{\alpha}.
	\end{equation}
	Together, \eqref{eq:elementary1}, \eqref{eq:elementary2}, \eqref{eq:jacobianlem3}, and \eqref{eq:jacobianlem4} imply 
	\begin{equation} \label{eq:jacobianlem5}
	\left| \frac{J_{F_{n,\xi}}(z)}{|\grad T_{\theta^n \xi}^n(z)e_x| | \grad T_{\xi}^{-n}(z)e_y|} - 1 \right| \le \frac{C}{\alpha}.
	\end{equation}
	
	From \eqref{eq:elementary1}, \eqref{eq:elementary2}, and \eqref{eq:jacobianlem5}, it remains only to show that 
	\begin{equation} \label{eq:jacobianlem6}
	\left| \frac{J_{u,n}(T_\xi^{-n}(z)) J_{s,n}(T^n_{\theta^n \xi}(z))}{|\grad T_{\theta^n \xi}^n(z)e_x| | \grad T_{\xi}^{-n}(z)e_y|} - 1 \right| \le \frac{C}{\alpha}.
	\end{equation}
	We will just prove that 
	\begin{equation} \label{eq:jacobianlem7}
	\left| \frac{J_{u,n}(T_\xi^{-n}(z))}{| \grad T_{\xi}^{-n}(z)e_y|} - 1 \right| \le \frac{C}{\alpha},
	\end{equation}
	as the same inequality for the other ratio follows similarly, and combining them to deduce \eqref{eq:jacobianlem6} is straightforward. Define $\bar{z} = T_\xi^{-n}(z)$ and let $\bar{e}_s$ and $\bar{e}_u$ denote the stable and unstable eigenvectors of $\grad T_\xi^n(\bar{z})$. Reasoning similar to the proof of \eqref{eq:jacobianlem1} above and noting also that $\bar{e}_s$ is the unstable eigenvector for $\grad T_\xi^{-n}(z)$ shows that
	\begin{equation} \label{eq:jacobianlem8}
	\left| \frac{J_{u,n}(\bar{z})}{|\grad T^n_\xi(\bar{z})\bar{e}_u|} - 1 \right| + \left| \frac{|\grad T_\xi^{-n}(z) \bar{e}_s|}{|\grad T_\xi^{-n}(z)e_y|} - 1\right| \le \frac{C}{\alpha}. 
	\end{equation}
	Now, since $T_\xi^n$ is area preserving, we have
	\begin{equation}\label{eq:jacobianlem9}
	|\grad T_\xi^n(\bar{z})\bar{e}_u| = |\grad T_\xi^n(\bar{z}) \bar{e}_s|^{-1} = |\grad T_\xi^{-n}(z) \bar{e}_s|.
	\end{equation}
	Combining \eqref{eq:jacobianlem8} with \eqref{eq:jacobianlem9} and then applying \eqref{eq:elementary2} yields \eqref{eq:jacobianlem7}.
\end{proof}
	
	\phantomsection

        \noindent {\bf Acknowledgments:} The authors thank the NSF for
               its support through the RTG grant
               DMS-2038056. T.E. acknowledges funding from the NSF
               grants DMS-2043024 and DMS-212474 as well as an Alfred
               P. Sloan Fellowship. T.E. thanks Craig Chen for very
               helpful computations and numerics related to this
               problem in the 2021-2022 academic year. All of the
               authors thank Gautam Iyer for useful feedback on an
               early draft of the paper and informative and
               illuminating discussions around this topic over the years.
	\addcontentsline{toc}{section}{References}
	\bibliographystyle{abbrv}
	\bibliography{WedgeBib}
	
\end{document}